\documentclass{article}
\usepackage{url}
\usepackage{amsfonts}
\usepackage{amsmath}
\usepackage{amsthm}
\usepackage{amssymb}
\usepackage{graphicx}
\usepackage{caption}
\usepackage{subcaption}
\usepackage{verbatim}
\usepackage{algorithmic}
\usepackage{algorithm}
\usepackage{mathtools}
\usepackage{color}
\usepackage{eurosym}
\usepackage{lineno}

\usepackage{amsfonts}

\nolinenumbers

\newcommand{\algrule}[1][.2pt]{\par\vskip.5\baselineskip\hrule height #1\par\vskip.5\baselineskip}
\newtheorem{theorem}{Theorem}

\newtheorem{setting}[theorem]{Setting}

\newtheorem{corollary}[theorem]{Corollary}

\newtheorem{definition}[theorem]{Definition}

\newtheorem{lemma}[theorem]{Lemma}

\newtheorem{proposition}[theorem]{Proposition}
\newtheorem{remark}[theorem]{Remark}

\newcommand{\Var}{\mathrm{Var}}
\newcounter{counter_summary}
\input{tcilatex}
\begin{document}

\title{Existence of Noise Induced Order, a Computer Aided Proof}
\author{Stefano Galatolo\thanks{
Dipartimento di Matematica, Universita di Pisa, Via Buonarroti 1, Pisa -
Italy. Email: stefano.galatolo@unipi.it}, Maurizio Monge\thanks{
Instituto de Matem\'{a}tica, Univ. Fed. do Rio de Janeiro, Av. Athos da Silveira Ramos 149, Bloco C
Cidade Universit\'{a}ria, Rio de Janeiro - Ilha do Fund\~{a}o - Brazil.
Email: maurizio.monge@im.ufrj.br}, Isaia Nisoli\thanks{
Instituto de Matem\'{a}tica, Univ. Fed. do Rio de Janeiro, Av. Athos da Silveira Ramos 149, Bloco C
Cidade Universit\'{a}ria, Rio de Janeiro - Ilha do Fund\~{a}o - Brazil.
Email: nisoli@im.ufrj.br}}
\maketitle

\begin{abstract}

We prove the existence of Noise Induced Order in the Matsumoto-Tsuda model, where it was originally discovered in 1983 by numerical simulations. This is a model of the famous Belosouv-Zabotinsky reaction, a chaotic chemical reaction, and consists of a one dimensional random dynamical system with additive noise. The simulations showed that an increase in amplitude of the noise causes the Lyapunov exponent to decrease from positive to negative; we give a mathematical proof of the existence of this transition.
The method we use relies on some computer aided estimates providing a certified approximation of the stationary measure in the $L^{1}$ norm. This is realized by explicit functional analytic estimates working together with an efficient algorithm. The method is
  general enough to be adapted to any piecewise differentiable dynamical system on the unit interval with additive noise. We also prove that the stationary measure of the system varies in a Lipschitz way if the system is perturbed and that the Lyapunov exponent of the system varies in a H\"older way when the noise amplitude increases.
\end{abstract}

\section{Introduction}


The \textquotedblleft Noise induced order\textquotedblright\ phenomenon was
discovered in numerical simulations and experiments regarding systems
modeled by a deterministic dynamics perturbed by noise. The somewhat
surprising phenomenon emerging is that the system appears to be chaotic for
very small noise intensity, but when the intensity increases the system
begins to have a less and less chaotic behavior, passing from a positive
Lyapunov exponent to a negative one.{
A similar behavior was also found for other indicators of chaos. This  paper however will focus only on the Lyapunov exponent.}

The phenomenon was first discovered by numerical simulations in \cite{MT}
in a system related to the famous Belousov-Zhabotinsky reaction (see Figure %
\ref{BZfig}) modeled by a one dimensional map perturbed by additive noise.
Real experiments confirmed the existence of the phenomenon appearing in the
model of the reaction (see \cite{Yal}, and also \cite{W}, \cite{YYK}, \cite%
{YYSKMN}, \cite{HE}, \cite{D}, \cite{M}, for some examples of related works).

\begin{figure}[htbp]
\centering
\includegraphics[height=2.9cm]{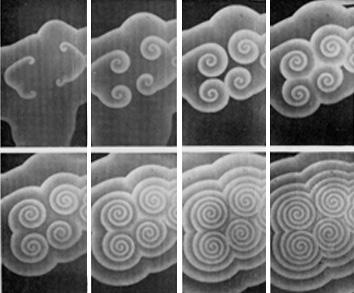}
\caption{A sequence of pictures showing various stages of the evolution of
the Belousov-Zhabotinsky reaction (from \protect\cite{ZZ}).}
\label{BZfig}
\end{figure}

Despite the impact that the discovery of such noise induced phenomena had in
the nonlinear science and physical literature (more than 390 citations to 
\cite{MT} on Google Scholar at the moment of writing this paper) to the best
of our knowledge there is no mathematical literature about Noise Induced
Order or rigorous proofs of its actual existence in nontrivial systems.

The mathematical study of this phenomenon is difficult because in the
deterministic part of the dynamics (see Figure \ref{f0}) there is a
coexistence of strongly expanding and strongly contracting regions and the
prevalence of expanding or contracting behavior for typical orbits is a
consequence of the global structure of the dynamics. \ We remark that this
phenomenon is one dimensional and inherently nonlinear, thus mathematically
not much related to the noise induced stabilization studied in \cite{ACV}
and following papers. With the help of some computer aided estimates we
prove that the global structure of this random system, allows expansion to prevail when the noise amplitude is very small, but the appearance of quite a large noise
allows the contraction to prevail. 

Our approach is based on the fact that the presence of the noise simplifies
the functional analytic properties of the transfer operator associated to
the system, smoothing out fine resolution details, and making it well
approximable by a finite resolution and finite dimensional one. This makes a
computer aided proof possible, letting the computer manage the complexity of
the deterministic part of the system at a \emph{finite} resolution scale and
understanding the global structure of the dynamics. However, the
computational power required to perform these computations in a naive way is
out of the range of current computers. This is true mostly for proving that
the Lyapunov exponent is positive in some case of very small noise range.
Indeed small noise corresponds to high resolution needed in the study of the
system. Because of this we had to find some mathematical clever way to
perform the needed estimates, using different functional spaces. This is the
main part of the mathematical work contained in what follows and can be
applied to many other dynamical systems perturbed by noise. The algorithm we
develop in this work was indeed already used in \cite{arnold} and \cite{itineracy} for the study of other dynamical systems with additive noise
considered as models of certain phenomena in climate science and
neuroscience.

It is known that naive computer simulations of chaotic systems may not be
reliable in some case (see for example \cite{Sim},\cite{Tuk},\cite{GNR},\cite%
{Gui},\cite{GGa} for examples of misleading naive simulations and a general
discussion on the problem). Beside the pure mathematical interest of a
rigorously proved result and a rigorously certified estimate, the study of
inherently reliable methods for the numerical study of chaotic dynamical
systems is strongly motivated.

{\noindent \textbf{Overview of the results.}} In this work we consider the
model of the Belousov-Zhabotinsky reaction studied in \cite{MT} (see also 
\cite{Z} for explanations on how the model can be deduced from the chemical
mechanism of the BZ reaction). This is a random dynamical system: a
deterministic map with additive noise at each iteration. The deterministic
part of the dynamics in the model is driven by a map $T_{a,b,c}:[0,1]%
\rightarrow \lbrack 0,1]$ defined by 
\begin{equation}
T_{a,b,c}(x)=\left\{ 
\begin{array}{c}
(a+(x-\frac{1}{8})^{\frac{1}{3}})e^{-x}+b,~~~0\leq x\leq 0.3 \\ 
c(10xe^{\frac{-10x}{3}})^{19}+b~~~0.3<x\leq 1%
\end{array}%
\right.  \label{Tabc}
\end{equation}
where 
\begin{equation}
a\in %
\scalebox{0.77}{$0.506073569036822351319599371053047956980141736828203749380990114218225638827_6^9$}%
,  \label{eqa}
\end{equation}%
\begin{equation}
b\in %
\scalebox{0.77}{$0.02328852830307032054478158044023918735669943648088852646123182739831022528_{158}^{213}$}%
,  \label{eqb}
\end{equation}%
\begin{equation}
c\in %
\scalebox{0.77}{$0.121205692738975111744666848150620569782497212127938371936404761693002104361_5^8$}%
.  \label{eqc}
\end{equation}%
The graph of an example of $T_{a,b,c}$ is shown in Figure \ref{f0}.
Following the \emph{Inverval Arithmetics} formalism we represent intervals
in the real line by subscript and superscript describing the decimal
expansion of lower and upper bounds for $x$ so that $x\in 0.klm_{tuv}^{xyz}$
means that $x$ belongs to the interval $[0.klmtuv,0.klmxyz]$. The choice of $%
a,b,c$ follow the one made in \cite{MT}, adding some more precision (see
Remark \ref{abcrm} for more details).

\begin{remark}
The interval arithmetics and its certified numerical methods (see \cite%
{Tucker} for an introduction) allow to obtain rigorous results as output of
the computer aided estimates. Our computer aided estimates are implemented
in this framework. We used SAGE \cite{SAGE} and the validated numerics
packages shipped with it. (The interval package is a binding to MPFI \cite%
{MPFI}.) Part of the numerical linear algebra was done using OpenCL \cite%
{OpenCL} running on Nvidia graphic cards.
\end{remark}

At each iteration of the map a uniformly distributed noise perturbation with
span of size $\xi $ is applied. Further details on the system are presented
in Section~\ref{map}.

\begin{figure}[htbp]
\centering
\includegraphics[height=4cm, width=5cm]{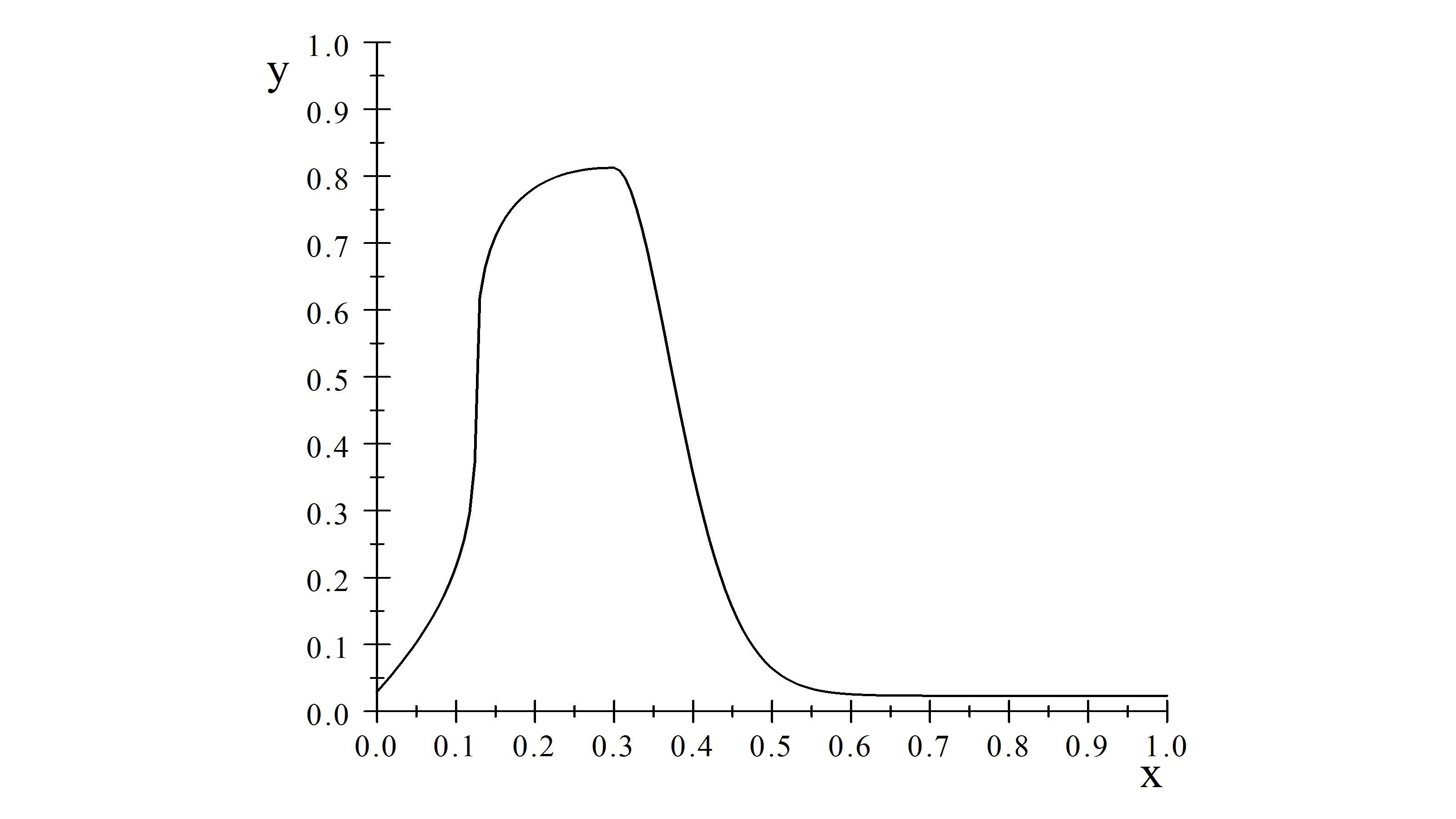}
\caption{The map $T_{a,b,c}$.}
\label{f0}
\end{figure}

In the paper we prove that when the noise size $\xi $ is contained in the
interval $[\xi _{1},1/2]$ where $\xi _{1}=\frac{8.73}{10^{5}}$ this random
system has a unique ergodic absolutely continuous stationary measure $\mu
_{\xi }$ (see Proposition \ref{ergo}) and consider the associated \emph{%
Lyapunov exponent} 
\begin{equation*}
\lambda _{\xi }:=\int_{0}^{1}\log |T^{\prime }(x)|d\mu _{\xi }.
\end{equation*}

We prove that the behaviour of $\lambda _{\xi }$ in the system is similar to
the one found by the numerical investigations of Matsumoto and Tsuda (\cite%
{MT}). \emph{In particular there is a transition from positive to negative
exponent as the noise amplitude increases.} We provide explicit examples of
values for the noise amplitude having positive and negative Lyapunov
exponent. In particular, the findings of the present work applied to the
Belousov-Zhabotinsky model defined above allow to state the following

\begin{theorem}
\label{theo:final} Let $\lambda _{\xi }$ be the Lyapunov expontent of the
system defined above\footnote{%
For every choice of the coefficients $a,b,c$ of $($\ref{Tabc}$)$ as in $($\ref{eqa}$)$,$($\ref{eqb}$)$,$($\ref{eqc}$)$.} with noise of size $\xi $. \ For each $%
\alpha <1,\ \lambda _{\xi }$ is $\alpha -$H\H{o}lder continuous as a
function of $\xi $ when $\xi \in \lbrack \frac{8.73}{10^{5}},1/2]$;
furthermore for $\xi _{1}=\frac{8.73}{10^{5}}$ and $\xi _{2}=\frac{8.60}{%
10^{3}}$ it holds

\begin{itemize}
\item[I1] the Lyapunov exponent $\lambda _{\xi _{1}}\in \lbrack \frac{8.365}{%
10^{2}},\frac{8.917}{10^{2}}]$, hence it is rigorously certified to be
positive;

\item[I2] the Lyapunov exponent $\lambda _{\xi _{2}}\in \lbrack \frac{%
-6.03602}{10},\frac{-6.03536}{10}]$, hence it is rigorously certified to be
negative.
\end{itemize}

Therefore, the system exibits Noise Induced Order.
\end{theorem}

We refer to Section \ref{results} for the certified results proving the
change of sign for the Lyapunov exponent, while we refer to Section \ref%
{sec7} (see Corollary \ref{holdlyapcor}) for the H\H{o}lder regularity of
the Lyapunov exponent. In the paper, we estimate similarly the Lyapunov
exponent for many other values of the noise size. The results are reported
in Table \ref{t1} and described in Figure~\ref{flyap} by a graph.

The method of proof and certification of these results relies on the
approximation of the stationary measure $\mu _{\xi }$ of the system up to a
certified error in the $L^{1}$ norm\footnote{%
The computation of stationary measures for dynamical systems perturbed by
noise was approached from different points of view in \cite{GI} where the
convergence (without effective bounds on the approximation error) of an
approximation scheme based on Fourier analysis was proved for certain
classes of maps, and in \cite{R}, where the computability of the stationary
measure up to a given error was considered in an abstract framework, giving
bounds on the computational complexity of the problem. The rigorous
computation of the stationary measure for expanding and contracting Iterated
Function Systems is considered in \cite{GMN}.}. Using this approximation we
can compute the Lyapunov exponent of the system up to a small certified
error. The methods used however are general enough to be applied to any
system formed by piecewise differentiable maps of the interval perturbed by
additive noise having a Bounded Variation distribution (see Sections \ref%
{map} and \ref{post} for a more detailed discussion).

{\noindent \textbf{Plan of the paper.}} In Section \ref{map} we describe the
systems under study and we introduce the technique we use for the
approximation of the transfer operator. The techniques leading to the proof
of Theorems \ref{theo:final} are explained in a sequence of settings of
decreasing generality where the needed assumptions on the system are listed
(see Settings \ref{set1}, \ref{set2}). In Section~\ref{firstalgo} we
describe how to find an explicit bound on the approximation error for the
computation of the stationary measure. In Subsection \ref{post} we describe
an efficient procedure which exploits information coming from a coarse
knowledge of the stationary measure in a bootstrap argument. This procedure
uses several technical lemmas estimating the variation of certain densities
and the norms of certain operators; such lemmas are listed in Section \ref%
{sec3} and proved in Section \ref{appendix} to make reading easier. In
Section \ref{coarsefine} we show an efficient procedure (which is used in
the main algorithm) for the estimation of the rate of contraction of a
finite rank transfer operator, when applied to zero average measures. This
is a quantitative measure of the rate of convergence to equilibrium of the
system which is involved in the quantitative estimate of its stability under
perturbation. This is important in the estimation of the approximation
error. In Section \ref{Lyap} we show the estimates needed to compute the
Lyapunov exponent of the system once we know the stationary measure. In
Section \ref{compresults} we apply all these techniques to the system
described in Section \ref{map}, showing the results of our computer aided
estimates, and proving the existence of noise induced order (and in
particular, Theorem \ref{theo:final}). In Section \ref{sec7} we consider the
stability of the stationary measure to changes in the system's parameters
and of the Lyapunov exponent on changes of the noise amplitude proving the H%
\H{o}lder continuity. Section \ref{sec:rds} contains some definitions and
generalities about Random Dynamical Systems we include for completeness and
to justify the correctness of the notion of Lyapunov exponent which is
estimated in the paper.

\paragraph{Aknowledgements.}

The authors thank Yuzuru Sato for the nice seminar where we learned about
the existence of Noise Induced Order and further explanations, EU
Marie-Curie IRSES ``Brazilian-European partnership in Dynamical Systems''
(FP7-PEOPLE-2012-IRSES 318999 BREUDS) for support during the research and
The Abdus Salam International Centre for Theoretical Physics (ICTP) where
the work started. I.N. was partially supported by CNPq, University of
Uppsala and KAW grant 2013.0315. I. N. thanks UFRJ, CAPES (through the programs PROEX and the CAPES-STINT project "Contemporary topics in non uniformly hyperbolic
dynamics"). S.G. thanks Leverhulm Trust grant (IN-2014-021, “Statistical properties of non uniformly hyperbolic dynamical systems: computer assisted proofs and rigorous computation”) and  GNAMPA/INDAM (“Stabilita e instabilita in sistemi con rumore, un approccio computer assistito.”)



\section{The system and its transfer operator\label{map}}


In this section we describe more precisely the system which will be studied
in the paper and the associated transfer operators. Basic
notions on random dynamical systems, stationary measures and Lyapunov
exponents we use in the paper are presented in Section \ref{sec:rds}.

A random dynamical system with additive noise on $[0,1]$ and reflecting
boundary conditions is a random perturbation of a deterministic map, defined
by 
\begin{equation}
x\rightarrow T(x)\hat{+}\omega _{n}  \label{systm}
\end{equation}%
where $T:[0,1]\rightarrow \lbrack 0,1]$ is a Borel measurable map and $%
\omega _{n}$ is an i.i.d. process distributed according to a probability
density $\rho _{\xi }$ and $\hat{+}$ is the "reflecting boundaries sum" on $%
[0,1]$ defined as follows.

\begin{definition}
Let $\pi :{\mathbb{R}}\rightarrow [0,1]$ be the piecewise linear map 
\begin{equation}
\pi (x)=\min_{i\in \mathbb{Z}}|x-2i|.  \label{ppi}
\end{equation}%
Let $a,b\in {\mathbb{R}}$ then 
\begin{equation*}
a\hat{+}b:=\pi (a+b)
\end{equation*}
where $+$ is the usual sum operator on $\mathbb{R}$. By this $a\hat{+}b\in
[0,1].$
\end{definition}

In the following we will consider the case where the noise density $\rho_\xi 
$ is the rescaling of some Bounded Variation kernel $\rho \in
BV[-\frac12,\frac12]$ with $\int \rho =1$ in the interval $[-\frac{\xi}2,%
\frac{\xi}2]$, hence
\begin{equation*}
\rho _{\xi }(x)=\frac{1}{\xi }\rho (\frac{1}{\xi }x).
\end{equation*}
(see Definition \ref{dbv} for a recall on the definition of bounded
variation.)

As described in the introduction, the model studied in the present paper is
the one studied in \cite{MT}. We consider a random dynamical system with
additive noise, as in \eqref{systm} where $T=T_{a,b,c}$ is defined in %
\eqref{Tabc}. At each iteration of the map a uniformly distributed noise
perturbation with span of size $\xi $ with reflecting boundaries is applied.

\begin{remark}
\label{abcrm} The parameters $a,b,c$ defined below \eqref{Tabc} have been
computed using interval arithmetic, in the implementation of our algorithm
they are represented by intervals. We explain the motivation for the choice
of the parameters in \cite{MT}. The parameter $c$ is defined in a way to be
nearby to a value for which $T(0.3^{-})=T(0.3^{+})$, making $T$ continuous
at $0.3$. The exact value of $c$ giving the continuity can be computed in a
closed form as: 
\begin{equation*}
c=\frac{20}{3^{20}\cdot 7}\cdot \bigg(\frac{7}{5}\bigg)^{1/3}\cdot
e^{187/10}.
\end{equation*}

The parameter $a$ is defined similarly in a way so that $T^{\prime}(0.3^-)=0$%
, making $T^{\prime }$ continuous at $0.3$. The value of such $a$ can be
computed in a closed form as: 
\begin{equation*}
a=\frac{19}{42}\cdot\bigg(\frac{7}{5}\bigg)^{1/3}.
\end{equation*}

The choice of the parameter $b$ in \cite{MT} is motivated by a parallel with
the logistic map. Let us denote by $T_b$ the map as only the parameter $b$
varies; each $T_b$ has a repelling fixed point $p_b$. In \cite{MT} this
explicit value of $b$ is chosen as an approximation of the parameter value
for which $T^4(0.3)=p_b$ following the kneading sequence ``RLLL'', i.e., a
Misiurewicz condition. We computed a certified interval containing $b$ using
a Newton Interval Method \cite{Tucker}. The interval enclosing $b$ is
computed giving the result shown at \eqref{eqb}.
\end{remark}

In our computer assisted estimates we will hence consider the map given at %
\eqref{Tabc} and uniform noise, however the mathematical treatment about
approximation of stationary measures in Section \ref{firstalgo} and
following is more general. We start considering a general random dynamical
system where $T$ is measurable and $\rho \in BV$ giving general results and
estimates we then improve using more assumptions on the system (i.e. $T$
piecewise smooth) putting ourselves in different general settings which are
stated precisely (see Settings \ref{set1}, \ref{set2}), to keep the
exposition as clear and general as possible.

%

\begin{remark}
The reflecting boundary condition at (\ref{systm}) is not influent when the
noise amplitude is smaller than the parameter $b$ (as it will be for all the
noise amplitudes considered in our computer aided proofs, see Table \ref{t1}%
).
\end{remark}

{\noindent \textbf{The transfer operator. }}We study the statistical
properties of the dynamical systems with additive noise, as defined at $($%
\ref{systm}$)$ through the study of the properties of their associated
transfer operators. Let us recall that a measurable map $T:X\rightarrow X$,
induces a map 
\begin{equation*}
L:SM(X)\rightarrow SM(X)
\end{equation*}%
where $SM(X)$ is the space of Borel signed measures on $X.$ The associated
map $L$ is defined in the following way: if $\mu \in SM(X)$ then: 
\begin{equation*}
L\mu (A)=\mu (T^{-1}(A)).
\end{equation*}%
In the literature, $L$ is also called the pushforward map associated to $T$,
sometime denoted by $T^{\ast }$. It is a linear operator on the vector space 
$SM(X)$ and  it is also called the transfer operator associated to $%
T $. The space of measures with density in $L^{1}([0,1])$ can be seen as a
subspace of $SM([0,1])$. If $T$ is nonsingular, $L$ can be considered as an
operator $L^{1}([0,1])\rightarrow L^{1}([0,1])$.

The (annealed) transfer operator $L_{\xi }$ associated to the system with
noise is given by the composition of the transfer operator $L$ and a
reflecting boundary convolution operator $N_{\xi }:SM([0,1])\rightarrow
L^{1}([0,1])$ (a suitable modification of the usual convolution), defined by 
\begin{equation}
N_{\xi }(f):=\rho _{\xi }\hat{\ast}f  \label{refl}
\end{equation}%
where the \textquotedblleft reflecting boundaries
convolution\textquotedblright\ $\hat{\ast}$ is defined similarly to the
reflecting boundaries sum as

\begin{definition}
\label{def:reflectingboundary} Let $\mu \in SM({\mathbb{R}})$. Let $\pi :{\ 
\mathbb{R}}\rightarrow [ 0,1]$ be the piecewise linear map defined at (\ref%
{ppi}) and $\pi ^{\ast }$ its associated pushforward map. We consider $\pi
^{\ast }\mu \in SM([0,1])$ as the \textquotedblleft reflecting
boundary\textquotedblright\ version of $\mu $.
\end{definition}

\begin{definition}
\label{def:hatconv} Let $f\in SM([0,1])$, $\rho _{\xi }\in BV[-{\xi },{\xi }%
] $. Let $\hat{f}\in SM(\mathbb{R})$ defined by $\hat{f}=1_{[0,1]}f$ and $%
\hat{\rho}_{\xi }\in L^{1}(\mathbb{R})$ by $\hat{\rho}_{\xi }=1_{[-{\xi },{%
\xi }]}\rho _{\xi }$. We define 
\begin{equation}
\rho _{\xi }\hat{\ast}f=\pi ^{\ast }(\hat{\rho}_{\xi }\ast \hat{f})
\label{hat}
\end{equation}%
where $\ast $ stands for the usual convolution operator on $\mathbb{R}$.
\end{definition}

This boundary reflecting convolution operator is regularizing, in particular 
$\rho _{\xi }\hat{\ast}f\in BV([0,1])$ if $f\in BV$, and has properties
similar to the usual convolution operator. For its basic properties see
Subsection \ref{summary}.

\begin{definition}
The \textbf{annealed transfer operator} $L_{\xi }:SM([0,1])\rightarrow
L^{1}([0,1])$ associated to a deterministic system with additive noise, as
described at (\ref{systm}) is defined as 
\begin{equation}
L_{\xi }:=N_{\xi }L.  \label{transfdef}
\end{equation}
\end{definition}

\begin{remark}
The annealed transfer operator is obtained by averaging the transfer
operator $L$ over all the possible noise perturbations. We refer to Section %
\ref{sec:rds} for some basic facts on this operator. In the following we
will mainly consider $L_{\xi }$ as an operator $L^{1}([0,1])\rightarrow
L^{1}([0,1])$. \ In the notation we emphasize the dependence of the operator
on the amplitude of the noise.
\end{remark}

\begin{definition}
Let $\mu _{\xi }\in L^{1}([0,1])$ be a fixed probability measure for $L_{\xi
}$, i.e, 
\begin{equation*}
L_{\xi }\mu _{\xi }=\mu _{\xi }
\end{equation*}%
we will call $\mu_{\xi}$ a \textbf{stationary measure for the system with
additive noise} or an \textbf{invariant measure for $L_{\xi }$}.
\end{definition}

By the regularizing properties of the convolution by a Bounded Variation
kernel, and standard compactness arguments it is easy to see that the
transfer operator $L_{\xi }$ corresponding to a map with additive noise has
at least one fixed point $f_{\xi }$ in $BV[0,1]$ (see \cite{GG}, Lemma 23
for more details). Following \cite{MT} we will study the Lyapunov exponent
of the system as $\xi $ varies. The Lyapunov exponent is defined as follows:

\begin{definition}
The \textbf{average Lyapunov exponent} associated to noise size $\xi$ is 
\begin{equation}
\lambda _{\xi }:=\int_{0}^{1}\log |T^{\prime }(x)|d\mu _{\xi }.  \label{li}
\end{equation}
\end{definition}

When $\mu _{\xi }$ is ergodic the average Lyapunov exponent coincides almost
everywhere with the pointwise Lyapunov exponent (see Section \ref{sec:rds}).
As a byproduct of our computer aided estimates, using the results given in
Section \ref{coarsefine} and \ref{sec7} we will prove that the systems
considered are ergodic (see Proposition \ref{ergo}). This justifies the
correctedness of the average Lyapunov exponent as an indicator of the
behaviour of the system. 

{\noindent \textbf{The Ulam Approximation.}} The main tool for the study of
the behavior of (\ref{li}) in this work is the rigorous approximation of the
stationary measure $\mu_{\xi}$. This is done by approximating $L_{\xi }$ by
a finite dimensional operator $L_{\delta ,\xi}:L^{1}([0,1])\rightarrow
L^{1}([0,1])$. The fixed points of $L_{\xi }$ are then approximated by the
ones of $L_{\delta ,\xi }$ with a certified bound on the approximation error.

Let $\pi _{\delta }:L^{1}([0,1])\rightarrow L^{1}([0,1])$ be a projection on
a finite dimensional space defined in the following way: the space $[0,1]$
is discretized by a partition $I_{\delta }$ (with $k$ elements); the
projection considered is defined by the conditional expectation 
\begin{equation}  \label{pidef}
\pi _{\delta }(f)=\mathbf{E}(f|F_{\delta })
\end{equation}
where $F_{\delta }$ is the $\sigma -$algebra associated to the partition $%
I_{\delta }.$

The approximated operator is then defined by finite element approach,
composing with $\pi$: 
\begin{equation*}
L_{\delta ,\xi }:=\pi _{\delta }N_{\xi }\pi _{\delta }L\pi _{\delta }.
\end{equation*}

The finite dimensional approximation of an operator based on the conditional
expectation, as above, is commonly called \emph{Ulam discretization} or 
\emph{Ulam method}. This method was widely studied in the literature (see,
e.g \cite{DelJu02},\cite{L},\cite{BM},\cite{BG},\cite{GN}).

Observe that 
\begin{equation*}
L_{\delta ,\xi }^{n}=(\pi _{\delta }N_{\xi }\pi _{\delta }L)^{n}\pi _{\delta
},
\end{equation*}%
taking into account that $\pi _{\delta }^{2}=\pi _{\delta }$. We remark that
since $||\pi _{\delta }||_{L^{1}\rightarrow L^{1}}\leq 1$ and $%
||L||_{L^{1}\rightarrow L^{1}}\leq 1$, then $||L_{\delta ,\xi
}||_{L^{1}\rightarrow L^{1}}\leq 1$.

\begin{remark}
Another discretization that could be used is 
\begin{equation*}
\tilde{L}_{\delta ,\xi }=\pi _{\delta }N_{\xi }L\pi _{\delta };
\end{equation*}%
while this definition is reasonable, it is more difficult to implement and
would force us to recompute the discretized operator for each size $\xi $ of
the noise. Our definition permits us to compute once and for all $\pi
_{\delta }L\pi _{\delta }$ which is computationally expensive but leads to a
sparse matrix, and then apply the operator $\pi _{\delta }N_{\xi }\pi
_{\delta }$ which is independent of the dynamics. 
\end{remark}



\section{Rigorous approximation of the stationary measure for dynamical systems with additive noise}
\label{sec3}


\label{firstalgo}

The certified approximation of the Lyapunov exponent is based on the
certified approximation of the stationary measure of the system in the $%
L^{1} $ norm. This is the main part of our general construction and is
described in this section. The algorithm to approximate the stationary
measure uses both \emph{a priori} (analytical) and \emph{a posteriori}
(computer assisted) estimates on the measure and on the transfer operator.%
\emph{\ For these estimates to be performed we do not need particular
expansion or hyperbolicity properties of the one dimensional map driving the
deterministic part of the dynamics.} We now introduce a general and simple
algorithm which works for measurable maps perturbed by additive noise with a
Bounded Variation kernel, in Section \ref{post} we refine this algorithm,
assuming that the map is piecewise smooth and getting much better estimates.
In this first part of the section we will hence work in the following
setting. 
\begin{setting}\label{set1}
Let us suppose $L_{\xi }$ is the transfer operator of a system with additive noise as considered in \eqref{systm}. We suppose the noise is distributed according to a Bounded Variation kernel $\rho_\xi$ with support in $[-\frac{\xi}2,\frac{\xi}2]$ and the deterministic part of the system is driven by a measurable map $T$.
\end{setting}

Let $L_{\delta ,\xi }$ be the Ulam approximation of $L_{\xi }$, defined by
projecting on a partition of size $\delta $. Let $f_{\delta ,\xi },$ $f_{\xi
}\in L^{1}$ respectively be invariant probability measures of $L_{\delta
,\xi }$ and $L_{\xi }$. Since the measure $f_{\delta ,\xi }$ is a fixed
point of the finite dimensional operator $L_{\delta ,\xi }$, it can be
computed to any precision. We will treat now the issue of estimating $\Vert
f_{\delta ,\xi }-f_{\xi }\Vert _{L^{1}}$ effectively.

\begin{lemma}
\label{contrle}Suppose that for some $\overline{n}\in \mathbb{N}$ 
\begin{equation}
||L_{\delta ,\xi }^{\overline{n}}|_{V}||_{L^{1}\rightarrow L^{1}}\leq \alpha
<1  \label{contreq0}
\end{equation}%
where $V=\{\nu \in L^{1},\nu ([0,1])=0\}$ being the space of zero-average
measures. Then 
\begin{equation}
\Vert f_{\xi }-f_{\xi ,\delta }\Vert _{L^{1}}\leq \frac{1}{1-\alpha }%
\left\Vert (L_{\delta ,\xi }^{\overline{n}}-L_{\xi }^{\overline{n}})f_{\xi
}\right\Vert _{L^{1}}.  \label{eqmain}
\end{equation}
\end{lemma}

\begin{remark}
We remark that $L_{\delta ,\xi }|_{V}$ is a finite dimensional operator and
can be represented by a matrix. Thus is possible for a computer to verify
that $||L_{\delta ,\xi }^{\overline{n}}|_{V}||_{L^{1}\rightarrow L^{1}}\leq
\alpha $ for some $\overline{n}$.
\end{remark}

\begin{proof}
(of Lemma \ref{contrle}) \ Since both $f_{\xi }$, $f_{\delta ,\xi }$ are
fixed points we can write 
\begin{eqnarray*}
\Vert f_{\delta ,\xi }-f_{\xi }\Vert _{L^{1}} &=&\Vert L_{\delta ,\xi
}^{n}f_{\delta ,\xi }-L_{\xi }^{n}f_{\xi }\Vert _{L^{1}} \\
&=&\Vert L_{\delta ,\xi }^{n}f_{\delta ,\xi }-L_{\delta ,\xi }^{n}f_{\xi
}+L_{\delta ,\xi }^{n}f_{\xi }-L_{\xi }^{n}f_{\xi }\Vert _{L^{1}} \\
&\leq &\Vert L_{\delta ,\xi }^{n}(f_{\delta ,\xi }-f_{\xi })\Vert
_{L^{1}}+\Vert (L_{\delta ,\xi }^{n}-L_{\xi }^{n})f_{\xi }\Vert _{L^{1}}
\end{eqnarray*}

Then 
\begin{equation*}
\Vert f_{\xi }-f_{\xi ,\delta }\Vert _{L^{1}}\leq \alpha ||f_{\xi }-f_{\xi
,\delta }\Vert _{L^{1}}+\left\Vert (L_{\delta ,\xi }^{\overline{n}}-L_{\xi
}^{\overline{n}})f_{\xi }\right\Vert _{L^{1}}
\end{equation*}%
implying the statement.
\end{proof}

\subsection{An informal description of the main algorithm to compute the
stationary measure up to a small given error. \label{infoalgo}}

Based on Lemma \ref{contrle}, a strategy to rigorously bound $\Vert
f_{\delta ,\xi }-f_{\xi }\Vert _{L^{1}}$ is the following. The computer will
find an $\overline{n}$ such that \eqref{contreq0} is satisfied, then %
\eqref{eqmain} will give an estimate for the approximation error. We remark
that if $\delta $ is small enough, $\left\Vert (L_{\delta ,\xi }^{\overline{n%
}}-L_{\xi }^{\overline{n}})f_{\xi }\right\Vert _{L^{1}}$ has a chance of
being small, since it is the difference of two nearby operators, both
applied to the same regular (Bounded Variation) measure. This is where the
size $\delta $ of the approximation grid has a role in the quality of the
approximation. On the other hand, $\overline{n}$ may depend on $\delta $.
This is why a priori estimates on the approximation error are not trivial.
Lemma \ref{contrle} \ provides some a posteriori estimate on the error; the
approximation error is known after the computer certifies the $\overline{n}$
and the $\alpha $ for which \eqref{contreq0} \ is satisfied.

Hence the main algorithm for the approximation of the invariant measure will
work as follows:

\begin{enumerate}
\item Given the grid size $\delta $, compute $L_{\delta ,\xi }$ and $f_{\xi
,\delta }$ up to some prescribed precision.

\item Find a good $\overline{n}$ and $\alpha $: we estimate $||L_{\delta
,\xi }^{n}|_{V}||_{L^{1}\rightarrow L^{1}}$ in an efficient way, finding a
good compromise between $\overline{n}$ and $\alpha $, in a way that $\alpha $
is not too close to $1$ and $\overline{n}$ not too big. We remark that the
norm of the finite dimensional operator $L_{\delta ,\xi }^{n}$ is directly
computable in principle, but if $\delta $ is small then the size of the
associated matrix is huge and this can be a hard computational task. For
this we use a \emph{coarse-fine} strategy which is explained in Section \ref%
{coarsefine}, and which takes into account that the huge matrix representing 
$L_{\delta ,\xi }$ is actually coming from a certain dynamical system with
noise.

\item Find a good estimate for $\left\Vert (L_{\delta ,\xi }^{\overline{n}%
}-L_{\xi }^{\overline{n}})f_{\xi }\right\Vert _{L^{1}}$. We remark that in
this formula $f_{\xi }$ is not known, but still we can find enough
information on it to estimate the difference of operators we are interested
in. This will be done by a method using both a priori and a posteriori
estimates, using an approximated knowledge of $f_{\xi }$ and its variation.
The procedure is explained in Section \ref{telescopicf} and in the following
sections. A simple but not efficient bound (equation \eqref{easy}) is proved
in \ref{init}; we refine the method in Section \ref{post} greatly improving
the efficency of the estimate with a bootstrap argument.

\item  Estimate the approximation
error $\Vert f_{\xi,\delta}-f_{\xi }\Vert _{L^{1}}$ by Lemma \ref{contrle}.

\end{enumerate}

\subsection{A bound for $\left\Vert (L_{\protect\delta,\protect\xi}^{%
\overline{n}}-L_{\protect\xi}^{\overline{n}})f_{\protect\xi}\right\Vert
_{L^{1}}$}

\label{telescopicf} Having outlined the main algorithm we now show how to
perform the main needed estimates. In this Section we describe how to
estimate the quantity appearing on the right hand side of \eqref{eqmain}.
This will be done by splitting this term in different parts which will be
treated differently. We start estimating the term as a telescopic sum whose
summands will be estimated in the following subsections.

\begin{lemma}
\label{ex4}Let $L_{\xi }$ the transfer operator of the random system and $%
L_{\delta ,\xi }$ its Ulam approximation, as defined in Section \ref{map}.
Let $f_{\xi }$ be an invariant probability measure for $L_{\xi }$, it holds 
\begin{align}
& \left\Vert (L_{\delta ,\xi }^{\overline{n}}-L_{\xi }^{\overline{n}})f_{\xi
}\right\Vert _{L^{1}}\leq \Vert (\pi _{\delta }-1)f_{\xi }\Vert
_{L^{1}}+\left( \sum_{i=0}^{\overline{n}-1}\left\Vert L_{\delta ,\xi
}|_{V}^{i}\right\Vert _{L^{1}\rightarrow L^{1}}\right) \times
\label{aposteriori_telescopic} \\
& \qquad \times \left( \left\Vert N_{\xi }(\pi _{\delta }-1)Lf_{\xi
}\right\Vert _{L^{1}}+\left\Vert N_{\xi }\pi _{\delta }L(\pi _{\delta
}-1)f_{\xi }\right\Vert _{L^{1}}\right)  \notag
\end{align}
\end{lemma}

\begin{proof}
The proof is based on a telescopic decomposition. We have%
\begin{eqnarray*}
L_{\delta ,\xi }-L_{\xi } &=&\pi _{\delta }N_{\xi }\pi _{\delta }L\pi
_{\delta }-N_{\xi }L \\
&=&N_{\xi }L-\pi _{\delta }N_{\xi }L \\
&&+\pi _{\delta }N_{\xi }L-\pi _{\delta }N_{\xi }\pi _{\delta }L \\
&&+\pi _{\delta }N_{\xi }\pi _{\delta }L-\pi _{\delta }N_{\xi }\pi _{\delta
}L\pi _{\delta }.
\end{eqnarray*}
Performing a similar decomposition to $L_{\delta ,\xi }^{\overline{n}}=(\pi
_{\delta }N_{\xi }\pi _{\delta }L)^{\overline{n}}\pi _{\delta }$. Pairing in
a suitable way the terms we inserted we obtain:

\begin{align*}
\left\Vert (L_{\delta,\xi}^{\overline{n}}-L_{\xi}^{\overline{n}})f_{\xi
}\right\Vert _{L^{1}} &= \left\Vert \big[(\pi _{\delta }N_{ \xi }\pi
_{\delta }L)^{\overline{n}}\pi _{ \delta }-(N_{\xi }L)^{\overline{n}}\big]%
f_{\xi }\right\Vert _{L^{1}} \\
&\leq \sum_{i=0}^{\overline{n}}\left\Vert (\pi _{\delta }N_{\xi }\pi
_{\delta }L)^{i}(\pi _{\delta }-1)(N_{\xi }L)^{n-i}f_{\xi }\right\Vert
_{L^{1}} \\
&\qquad+\sum_{i=0}^{\overline{n}-1}\left\Vert (\pi _{\delta }N_{\xi }\pi
_{\delta }L)^{i}\pi _{\delta }N_{\xi }(\pi _{\delta }-1)L(N_{\xi
}L)^{n-i-1}f_{\xi }\right\Vert _{L^{1}} \\
&=\sum_{i=0}^{\overline{n}}\left\Vert (\pi _{\delta }N_{\xi }\pi _{\delta
}L)^{i}(\pi _{\delta }-1)f_{\xi }\right\Vert _{L^{1}} \\
&\qquad+\sum_{i=0}^{\overline{n}-1}\left\Vert (\pi _{\delta }N_{\xi }\pi
_{\delta }L)^{i}\pi _{\delta }N_{\xi }(\pi _{\delta }-1)Lf_{\xi }\right\Vert
_{L^{1}}
\end{align*}%
considering that $f_{\xi }$ is fixed by $N_{\xi }L=L_\xi$. Shifting indexes
by $1$ in the the first sum, the estimate can be written as 
\begin{align}
&\left\Vert (L_{\delta,\xi}^{\overline{n}}-L_{\xi}^{\overline{n}})f_{\xi
}\right\Vert _{L^{1}} \leq  \notag \\
\leq &\|(\pi_\delta-1)f_\xi\|_{L^1} + \sum_{i=0}^{\overline{n}-1}\left\Vert
(\pi _{\delta }N_{\xi }\pi _{\delta }L)^{i}\pi_\delta|_V\right\Vert_{L^1\to
L^1}\cdot \left\Vert N_{\xi }\pi _{\delta }L(\pi _{\delta
}-1)f_{\xi}\right\Vert _{L^{1}}  \notag \\
&\qquad+\sum_{i=0}^{\overline{n}-1}\left\Vert (\pi _{\delta }N_{\xi }\pi
_{\delta }L)^{i}\pi _{\delta }|_V\right\Vert_{L^1\to L^1}\cdot \left\Vert
N_{\xi }(\pi _{\delta }-1)Lf_{\xi }\right\Vert _{L^{1}}  \notag \\
=& \|(\pi_\delta-1) f_\xi\|_{L^1} + \left(\sum_{i=0}^{\overline{n}%
-1}\left\Vert L_{\delta,\xi}|_V^{i}\right\Vert_{L^1\to L^1}\right)\times 
\notag \\
&\qquad\times \left( \left\Vert N_{\xi }(\pi _{\delta }-1)Lf_{\xi
}\right\Vert_{L^{1}} + \left\Vert N_{\xi }\pi _{\delta }L(\pi _{\delta
}-1)f_{\xi}\right\Vert_{L^{1}} \right)  \notag
\end{align}
(notice that $(\pi_\delta-1)g$ has always average zero for any $g$, and
consequently belongs to $V$).
\end{proof}

\subsubsection{An initial (a priori) bound for $\Vert f_{\protect\xi }-f_{%
\protect\xi ,\protect\delta }\Vert _{L^{1}} $}

\label{init}

Now we show a strategy to get a simple effective bound for the approximation
error $\Vert f_{\xi }-f_{\xi ,\delta }\Vert _{L^{1}}$ based Lemma \ref{ex4},
estimating the summands on the right hand side of (\ref%
{aposteriori_telescopic}) by quantities which are known from the description
of the system or can be computed by its approximated transfer operator. In
the next section we will improve the method, using more information on $T$
and $f_{\xi }$ and getting much more efficient estimates.

\begin{lemma}
\label{Lemmaeasy} Let $f_{\xi }$ a stationary measure for a system defined
as in (\ref{systm}) and $f_{\xi ,\delta }$ a stationary measure for its Ulam
approximation, as defined in Section \ref{map}. If there is $\overline{n}$
such that 
\begin{equation}
||L_{\delta ,\xi }^{\overline{n}}|_{V}||_{L^{1}\rightarrow L^{1}}\leq \alpha
<1  \label{contreq}
\end{equation}%
then 
\begin{equation}
\Vert f_{\xi }-f_{\xi ,\delta }\Vert _{L^{1}}\leq \frac{1+2\sum_{i=0}^{%
\overline{n}-1}C_{i}}{2(1-\alpha )}\,\delta \xi ^{-1}\,\mathrm{Var}(\rho ).
\label{easy}
\end{equation}%
where $0\leq C_{i}\leq 1$ are such that  $\Vert L_{\delta ,\xi }^{i}|_{V}\Vert
_{L^{1}\rightarrow L^{1}} \leq C_{i} $.
\end{lemma}

\begin{proof}
The proof of the lemma is based on the following estimates, proved in
Corollary \ref{c1} and Prop. \ref{prosop} (Section \ref{appendix}) allowing
a first estimate on the right hand side of \eqref{eqmain}. We have 
\begin{equation}  \label{sopra}
\Vert N_{\xi }(1-\pi _{\delta })\Vert _{L^{1}\rightarrow L^{1}}\leq \frac{1}{%
2}\delta \xi ^{-1}\mathrm{Var}(\rho ).
\end{equation}
\begin{equation}  \label{sotto}
\Vert (1-\pi _{\delta })N_{\xi }\Vert _{L^{1}\rightarrow L^{1}}\leq \frac{1}{%
2}\delta \xi ^{-1}\mathrm{Var}(\rho ).
\end{equation}
Now we apply (\ref{sopra},\ref{sotto}) to the summands of the righ hand side
of (\ref{aposteriori_telescopic}). We see that all the items there have
either a $N_\xi(1-\pi_\delta)$ or a $(1-\pi_\delta)N_\xi$ appearing. Indeed
since $||f_\xi||_1 \leq 1 $ 
\begin{equation*}
\|(\pi_\delta-1)f_\xi\|_{L^1}= \|(\pi_\delta-1)N_\xi L f_\xi\|_{L^1}\leq 
\frac{1}{ 2}\delta \xi ^{-1}\Var(\rho ).
\end{equation*}
Similarly the other summands, can be estimated recalling that $\Vert
\pi_\delta \Vert _{L^{1}\rightarrow L^{1}}\leq 1$ and $\Vert N_\xi \Vert
_{L^{1}\rightarrow L^{1}} \leq 1$. Applying (\ref{eqmain}) we get the
statement.
\end{proof}

The estimate given at \eqref{easy} mainly depend on the ratio $\delta \xi
^{-1}$ between the partition size and the noise amplitude. This estimate is
obtained without any information on the deterministic part of the dynamics,
only the information about the contraction rate of the approximated transfer
operator $L_{\delta ,\xi }$ (to obtain $\overline{n}$ and $\alpha $). This
would already allow to obtain a good approximation for the invariant density 
$f_{\xi }$, in principle, if we had enough computation power to carry on the
computation with a very small $\delta $. Unfortunately, in the
Matsumoto-Tsuda system, positive Lyapunov exponent appears for very small
sizes of the noise making the computation unfeasible even with the help of a
supercomputer, due to the growth of the computational complexity as $\xi $
becomes small.

Therefore, we have to apply a more subtle and complicated strategy where the
bootstrap argument comes into play, i.e., using some information on $f_\xi$
(and in particular about its variation in given intervals) which we can
obtain with a preliminary computation.

\subsection{A stronger (a posteriori) bound}

\label{post}

In this section we analyze better \eqref{aposteriori_telescopic} and see how
using some more assumptions on $T$, more information on $f_{\xi }$ \ and the
use of the Wasserstein distance, we can drastically improve the estimates
given in Lemma \ref{Lemmaeasy}. We remark that the explicit error bound
provided by Lemma \ref{Lemmaeasy} is  proportional to $\delta \xi
^{-1};$ the new error estimate will be a sum where most summands are
proportional to $\delta ^{2}\xi ^{-1}$.

\begin{setting}\label{set2}  From now on we will suppose we are in the framework of Setting \ref{set1} and furthermore we suppose $T$ being piecewise smooth. We suppose that there is a partition $\{P_{i}\}_{1\leq
i\leq k}$ such that 

\begin{itemize}
\item each $P_{i}$ is an interval,

\item on each $P_{i}$ the branch $T_{i}:=T|_{P_{i}}$ is monotonic and $C^{2}$
in the interior of $P_{i}$

\item The limits of $T_{i}^{\prime }(x)$ as $x$ tend to the frontier of $%
P_{i}$ exist in $\mathbb{R}\cup 	\{-\infty,\infty\}$.
\end{itemize}
\end{setting}

We will consider the transfer operator $L$ related to the map and to every
branch of the map. For each $L^1$ density $g$, we let $L_{i}g$ be the
component of $Lg$ coming from the $i$-th monotone branch (the pushforward
map related to $T_i$), that is 
\begin{equation}  \label{Li}
L_{i}g(x)=\left\{ 
\begin{array}{cc}
\frac{g(T_{i}^{-1}(x))}{T^{\prime }(T_{i}^{-1}(x))} & \text{ if }x\in
T_i(I_i), \\ 
0 & \text{ elsewhere.}%
\end{array}%
\right.
\end{equation}
In this way we have $Lg=\sum_i L_ig$.

Once the discretized transfer operator $L_{\delta ,\xi }$ is computed and
has a unique fixed probability measure $f_{\delta ,\xi },$ an approximation $%
\tilde{f}$ \ for $f_{\delta ,\xi }$ can be computed up to any given small
error in $L^{1}$ (this is the computation of the fixed point of the big
matrix representing $L_{\delta ,\xi }$). Let us assume that a numerical
approximation $\tilde{f}$ of $f_{\delta ,\xi }$ in the $L^{1}$ norm is
computed. Using $\tilde{f}$, we will look for a strategy for estimating the
total error $\Vert f_{\xi }-\tilde{f}\Vert _{L^{1}}$, that includes the
appoximation error $\Vert f_{\xi }-f_{\delta ,\xi }\Vert _{L^{1}}$ (because
we approximated on a partition of size $\delta $) and the numerical error $%
\Vert f_{\delta ,\xi }-\tilde{f}\Vert _{L^{1}}$. We remark that%
\begin{equation}
\Vert f_{\xi }-\tilde{f}\Vert _{L^{1}}\leq \Vert f_{\delta ,\xi }-\tilde{f}%
\Vert _{L^{1}}+\Vert f_{\xi }-f_{\delta ,\xi }\Vert _{L^{1}}.  \label{neweq1}
\end{equation}%
For the estimation of $\Vert f_{\delta ,\xi }-\tilde{f}\Vert _{L^{1}}$ in
our algorithm we apply the same method described in \cite{GN}. To find a stronger bound for $\Vert f_{\xi
}-f_{\xi ,\delta }\Vert _{L^{1}}$ let us start again from the estimate given
at Lemma \ref{ex4}. We will estimate independently the terms 
\begin{equation}
\Vert (\pi _{\delta }-1)f_{\xi }\Vert ,\quad \Vert N_{\xi }(\pi _{\delta
}-1)Lf_{\xi }\Vert _{L^{1}},\quad \Vert N_{\xi }\pi _{\delta }L(\pi _{\delta
}-1)f_{\xi }\Vert _{L^{1}},  \label{11}
\end{equation}%
appearing at (\ref{aposteriori_telescopic}), using the information we can
extract from the approximation $\tilde{f}$. Except in the first case (that
has the smallest weight in the estimate, according to (\ref%
{aposteriori_telescopic})), the estimate will become roughly proportional to 
$\delta ^{2}\xi ^{-1}$, greatly improving the quality of the approximation
certification. For each of the terms in \eqref{11} we prove in Sections \ref%
{estpm1f}, \ref{estpm1f2}, \ref{estpm1f3} bounds of the form 
\begin{equation*}
A||\tilde{f}-f_{\xi }||_{L^{1}}+B
\end{equation*}%
where $A$ and $B$ depend on $\delta $ and become small when $\delta $ is
small. We remark that these bounds \textbf{depend} on the error $||\tilde{f}%
-f_{\xi }||_{L^{1}}$ itself. This together with precise bounds, based on
some approximation $\tilde{f}$, permits us to tighten the bounds on the
error $||\tilde{f}-f_{\xi }||_{L^{1}}$, using a so called \textquotedblleft
bootstrapping\textquotedblright\ process.

\subsubsection{A summary of norms and estimates}

\label{summary}

Before entering in the details of the estimate of \ref{11}, we introduce
some of the norms used in the paper, and we summarize some of the bounds
that are used in this section and proved in Appendix \ref{appendix2}.

\begin{definition}
\label{dbv}Let $X\subset \lbrack 0,1]$ be a finite union of pairwise
disjoint intervals, $X=\bigcup_{j}I_{j}$, $I_{i}\cap I_{j}=\emptyset $ for $%
i\neq j$. We define the \textbf{variation on $X$} of the function $f$ (and
denote it by $\mathrm{Var}_{X}(f)$) as follows:

\begin{itemize}
\item when $X$ is an interval (the endpoints may be included or not), the
variation is defined as 
\begin{equation*}
\mathrm{Var}_X(f):=\sup_{\{x_0<x_1<\dots<x_k\}\subset X}\sum_{i=0}^{k-1}
|f(x_{i+1})-f(x_i)|
\end{equation*}
(the supremum being over finite increasing sequences of any length $k$
contained in $X$);

\item when $X$ is a finite union of pairwise disjoint intervals $X=\bigcup_j
I_j$, $I_i\cap I_j=\emptyset$ for $i\neq j$, the variation is defined as $%
\mathrm{Var}_X(f):= \sum_j \mathrm{Var}_{I_j}(f)$.
\end{itemize}
\end{definition}

\begin{remark}
As it is well known, if $f$ has a weak derivative, then $\mathrm{Var}_X(f) =
\|f^{\prime }\|_{L^1(X)}$, see for instance \cite[Chap. 4, Prop. 4.2]{SS}.
\end{remark}

We will also consider a norm which is weaker than the $L^1$ norm.

\begin{definition}
\label{def:Wasserstein} Let $f$ be a function in $L^{1}([0,1])$ with zero
average, we define the \textbf{Wasserstein-like norm} of $f$, as 
\begin{equation}
\Vert f\Vert _{W}:=\Vert F\Vert _{L^{1}},\text{ where }F(x)=\int_{0}^{x}f(t)%
\mathrm{d}t.
\end{equation}%
Let $Y\subset \lbrack 0,1]$ be a finite union of pairwise disjoint intervals 
$Y=\bigcup_{j}I_{j}$, $I_{i}\cap I_{j}=\emptyset $ for $i\neq j$, we denote
by $W(Y)$ the space of \ $L^{1}$ functions with support contained in $Y$
having zero average in each $I_{j}$. If $f\in W(Y)$ we define its norm by 
\begin{equation}
\Vert f\Vert _{W(Y)}:=\Vert f\Vert _{W}.
\end{equation}
\end{definition}


Next proposition contains a summary of the bounds used in next proofs and
the location of their proof in the paper; it can be used as an handy
guideline throughout the paper, we refer to the cited lemmas and proposition
for details.

\begin{proposition}[Summary of the bounds]
\label{summmary} Let $\pi_{\delta}$ be the Ulam projection on a homogeneous
partition of size $\delta $, as defined in (\ref{pidef}), let $I$ be a set
which is a finite union of intervals of the partition then:

\begin{enumerate}
\item $||1-\pi_{\delta}||_{\mathrm{Var} \to L^1}\leq \delta/2$, Lemma \ref%
{Lemma3},

\item $||1-\pi_{\delta}||_{L^1\to W}\leq \delta/2$, Lemma \ref{lem1mpLW},

\item $||1-\pi_{\delta}||_{\mathrm{Var}(I)\to W(I)}\leq \delta^2/8$, Lemma %
\ref{idmpVarW}. \setcounter{counter_summary}{\value{enumi}}
\end{enumerate}

Let $N_{\xi}$ be the convolution operator then:

\begin{enumerate}
 \setcounter{enumi}{\value{counter_summary}}

\item $||N_{\xi}||_{L^1\to \mathrm{Var}}\leq \xi^{-1}\mathrm{Var}(\rho)$,
Lemma \ref{chepallepero},

\item $||N_{\xi}||_{W\to L^1}\leq \xi^{-1}\mathrm{Var}(\rho)$, Lemma \ref{L1}%
,

\item $||(1-\pi_{\delta})N_{\xi}||_{L^1\to L^1}\leq \frac{1}{2}\delta
\xi^{-1}\mathrm{Var}(\rho)$, Corollary \ref{c1},

\item $||N_{\xi}(1-\pi_{\delta})||_{L^1\to L^1}\leq \frac{1}{2}\delta
\xi^{-1}\mathrm{Var}(\rho)$, Proposition \ref{prosop}. %
\setcounter{counter_summary}{\value{enumi}}
\end{enumerate}

Let $L$ be the transfer operator associated to $T$, and let $L_{i}g$ the
component of $Lg$ coming from the $i-th$ branch as defined at (\ref{Li}),
then:

\begin{enumerate}
\setcounter{enumi}{\value{counter_summary}}

\item $||L||_{W(I)\to W}\leq ||T^{\prime }||_{L^{\infty}(I)}$, Lemma \ref%
{LemmaW}

\item in Lemma \ref{varLf} the \textbf{local variation inequality} is
proved: 
\begin{align*}
\mathrm{Var}_{I}(L_{i}g)\leq& \mathrm{Var}_{T_{i}^{-1}(I)}(g)\cdot
\left\Vert \frac{1}{T^{\prime }}\right\Vert _{L^{\infty
}(T_{i}^{-1}(I))}+\Vert g\Vert_{L^{1}(T_{i}^{-1}(I))}\cdot \left\Vert \frac{%
T^{\prime\prime}}{T^{\prime 2}}\right\Vert _{L^{\infty }(T_{i}^{-1}(I))} \\
&+\sum_{y\in \partial Dom(T_i):T(y)\in I}\left|\frac{g(y)}{T^{\prime }(y)}%
\right|.
\end{align*}
\end{enumerate}
\end{proposition}

\subsubsection{Estimate for $\|(\protect\pi_\protect\delta-1)f_\protect\xi%
\|_{L^1}$}

\label{estpm1f}

We give now an estimate for the first item of (\ref{11}).

\begin{lemma}
Let $\pi_{\delta}$ be the Ulam projection on a homogeneous partition of size 
$\delta$, then 
\begin{equation}  \label{ab1}
\Vert(\pi _{\delta }-1)f_{\xi}\Vert_{L^1} \leq A_1 \cdot \Vert f_{\xi}-%
\tilde{f}\Vert_{L^1} + B_1
\end{equation}
for 
\begin{equation}  \label{ab1def}
A_1 = \frac{\delta}{2}\xi^{-1}\mathrm{Var}(\rho),\qquad B_1 = \frac{\delta}{2%
} \cdot \mathrm{Var}(N_\xi L\tilde{f}).
\end{equation}
\end{lemma}

\begin{proof}
We estimate: 
\begin{align*}
\Vert(\pi _{\delta }-1)f_{\xi}\Vert_{L^1} &= \Vert(\pi _{\delta }-1)N_\xi L
f_{\xi}\Vert_{L^1} \\
\leq &\ \Vert(\pi _{\delta }-1)N_\xi L (f_{\xi}-\tilde{f})\Vert_{L^1} +
\Vert(\pi _{\delta }-1)N_\xi L \tilde{f}\Vert_{L^1} \\
\leq &\ \Vert(\pi _{\delta }-1)N_\xi\Vert_{L^1}\cdot \Vert f_{\xi}-\tilde{f}%
\Vert_{L^1} + \Vert\pi _{\delta }-1\Vert_{\Var\rightarrow L^1} \cdot \Var%
(N_\xi L\tilde{f}) \\
\leq &\ \frac{\delta}{2}\xi^{-1}\Var(\rho) \cdot \Vert f_{\xi}-\tilde{f}%
\Vert_{L^1} + \frac{\delta}{2} \cdot \Var(N_\xi L\tilde{f})
\end{align*}
where in the last line we used Proposition \ref{summmary}, items 6) and 1).
\end{proof}

An upper bound on $\mathrm{Var}(N_\xi L\tilde{f})$ will be estimated using
the results in Section \ref{varNLf} and the explicit knowledge of the
computed $\tilde{f}$.

\subsubsection{Estimate for $||N_{\protect\xi }(1-\protect\pi _{\protect%
\delta })Lf_{\protect\xi }||_{L^1}$}

\label{estpm1f2} We give now the estimate of the second item of (\ref{11}).
The main idea is to use a coarser partition $\Pi$ made of intervals whose
size is an integral multiple of $\delta$. Then, instead of estimating $%
||N_{\xi }(1-\pi _{\delta })Lf_{\xi }||_{L^1}$ globally we bound this
quantity on each interval of $\Pi $ exploiting the approximate knowledge of
the variation of $f_{\xi }$ in the interval. This drastically improves the
quality of the estimate in almost every interval of $\Pi$ and allows the
error-checking computation to be performed on a partition which is coarser
than the initial partition of size $\delta$. 

\begin{lemma}
\label{lemma_N1pLf} Let $\Pi$ be a uniform partition whose parts have size
that is an integral multiple of $\delta$, $T$ piecewise monotonic with $L_i$
defined as above. We have 
\begin{equation}  \label{ab2}
\Vert N_{\xi }(1-\pi _{\delta })Lf_{\xi }\Vert _{L^{1}}\leq A_2 \cdot \Vert
f_{\xi }-\tilde{f}\Vert _{L^{1}} + B_2,
\end{equation}
with 
\begin{align}  \label{ab2def}
A_2 &= \frac{\delta }{2}\mathrm{Var}(\rho _{\xi })  \notag \\
B_2 &= \frac{\delta }{2}\mathrm{Var}(\rho _{\xi })\sum_{I\in \Pi
}\sum_{i}\min \left\{ \frac{\delta }{4}\cdot \mathrm{Var}_{I}(L_{i}\tilde{f}%
),\Vert L_{i}{\tilde{f}}\Vert _{L^{1}(I)}\right\}.
\end{align}
\end{lemma}

\begin{proof}
We can estimate as 
\begin{align}
\Vert N_{\xi }(1-\pi _{\delta })Lf_{\xi }\Vert _{L^{1}} &\ \leq \Vert N_{\xi
}(1-\pi _{\delta })L\tilde{f}\Vert _{L^{1}}+\Vert N_{\xi }(1-\pi _{\delta
})L(f_{\xi }-\tilde{f})\Vert _{L^{1}}  \notag \\
\leq &\ \Vert N_{\xi }(1-\pi _{\delta })L\tilde{f}\Vert _{L^{1}} +\Vert
N_{\xi }(1-\pi _{\delta })\Vert _{L^{1}}\cdot \Vert f_{\xi }-\tilde{f}\Vert
_{L^{1}}  \notag \\
\leq &\ \Vert N_{\xi }(1-\pi _{\delta })L\tilde{f}\Vert _{L^{1}} + \frac{%
\delta }{2}\Var(\rho _{\xi }) \cdot \Vert f_{\xi }-\tilde{f}\Vert _{L^{1}}
\label{deltaq_est1}
\end{align}
The first summand can be rewritten as 
\begin{align*}
\Vert N_{\xi }(1-\pi _{\delta })L{\tilde{f}}\Vert _{L^{1}} &\leq\Vert N_{\xi
}\Vert _{W\rightarrow L^{1}}\cdot \Vert (1-\pi _{\delta })L\tilde{f}\Vert_{W}
\\
&\leq \Vert N_{\xi }\Vert _{W\rightarrow L^{1}}\cdot \sum_{I\in\Pi}\Vert
(1-\pi _{\delta })L\tilde{f}\cdot \chi_I\Vert _{W(I)}
\end{align*}
splitting on the intervals of the partition, and using Lemma \ref{varLf}
(because each $I\in\Pi$ is a union of intervals of the partition of size $%
\delta$) 
\begin{align*}
\leq \Vert N_{\xi }\Vert _{W\rightarrow L^{1}}\cdot &\sum_{I\in \Pi
}\sum_{i}\min \bigg\{\Vert 1-\pi _{\delta }\Vert _{\Var_I\rightarrow
W(I)}\cdot \Var_{I}(L_{i}\tilde{f}), \\
& \Vert 1-\pi _{\delta }\Vert _{L^{1}(I)\rightarrow W(I)}\cdot \Vert L_{i}%
\tilde{f}\Vert _{L^{1}(I)}\bigg\} \\
\leq \Var(\rho _{\xi })\cdot& \sum_{I\in \Pi }\sum_{i}\min \left\{ \frac{%
\delta ^{2}}{8}\cdot \Var_{I}(L_{i}\tilde{f}),\frac{\delta }{2}\cdot \Vert
L_{i}{\tilde{f}}\Vert _{L^{1}(I)}\right\} .
\end{align*}
In the last step we used the property of the $W$ norm stated in Lemma \ref%
{lem1mpLW} and \ref{idmpVarW}. Adding the second term of \eqref{deltaq_est1}
we have proved the result.
\end{proof}

Therefore, once we have an approximation $\tilde{f}$ we can compute $A_2$, $%
B_2$ by a simple algorithm that evaluates the double summation in the last
equation of the above Lemma.

\begin{remark}
When computing $B_2$, the minimum could be always the one depending on $%
\Vert L_{i}{\tilde{f}}\Vert _{L^{1}(I)}$. If this happens, the sum adds up
to $1$, and the new bound is worse than the \emph{a priori} estimate, since
we are introducing a factor $\Vert f_{\xi }-\tilde{f} \Vert _{L^{1}}+1>1$.

In practice, this does not happen. On each $i$-th preimage of an
interval $I$ of the partition the new estimate will provide a better bound
as soon as 
\begin{equation*}
\frac{\delta }{4}\cdot \mathrm{Var}_{I}(L_{i}\tilde{f})<\Vert L_{i}\tilde{f}%
\Vert _{L^{1}(I)}.
\end{equation*}

First of all, the variation of $\tilde{f}$ has an a priori bound by 
\begin{equation*}
\mathrm{Var}(\tilde{f}) = \mathrm{Var}(\pi_\delta N_\xi \pi_\delta
L\pi_\delta \tilde{f}) \leq \|\pi_\delta\|_{\mathrm{Var}} \cdot
\|N_\xi\|_{L^1\rightarrow\mathrm{Var}}\cdot 1 \leq \xi^{-1}\mathrm{Var}(\rho).
\end{equation*}

Now, we can try to control $\mathrm{Var}_{I}(L_{i}\tilde{f})$  by using the local variation inequality in Proposition \ref{summmary}. If the preimage $T_i^{-1}(I)$ does not contain a critical point or a singular point, we expect $(\delta^2/8) \cdot \mathrm{Var}_{I}(L_{i}\tilde{f})$ to be small.

For the intervals $I$ where we cannot apply the local variation inequality or where it does not give us good enough bounds we  fall back to the a-priori estimate depending on the $L^{1}$ mass rather
than on the variation.
\end{remark}

\subsubsection{An estimate for $||N_{\protect\xi }\protect\pi _{\protect%
\delta }L(\protect\pi _{\protect\delta }-1)f_{\protect\xi }||_1$}

\label{estpm1f3} We give an estimate of the third item of (\ref{11}). The
general idea is similar to the one explained in the previous section, again,
some required estimates are technical lemmas proved in Section \ref{appendix}%
.

\begin{lemma}
\label{lemma_NpL1pf} Let $\Pi$ be a uniform partition whose parts have size
that is multiple of $\delta$, we have: 
\begin{equation}  \label{ab3}
\Vert N_{\xi }\pi _{\delta }L(1-\pi _{\delta })f_{\xi }\Vert _{L^{1}} \leq
A_3 \cdot \Vert f_{\xi }-\tilde{f}\Vert _{L^{1}} + B_3,
\end{equation}
with 
\begin{align}
A_3 =& \frac{\delta }{2}\xi^{-1}\mathrm{Var}(\rho),  \notag \\
B_3 =& \sum_{I\in \Pi }\min \left\{ \frac{\delta ^{2}}{8}\xi^{-1}\mathrm{Var}%
(\rho )\cdot \Vert T^{\prime }\Vert _{L^{\infty }(I)},\frac{\delta }{2}%
\right\} \mathrm{Var}_{I}(N_{\xi }L\tilde{f})  \label{ab3def} \\
&+\frac{\delta ^{2}}{4}\xi^{-1}\mathrm{Var}(\rho)\cdot \mathrm{Var}(N_{\xi }L%
\tilde{f}).  \notag
\end{align}
\end{lemma}

\begin{proof}
We have 
\begin{align}
\Vert & N_{\xi }\pi _{\delta }L(1-\pi _{\delta })f_{\xi }\Vert
_{L^{1}}=\Vert N_{\xi }\pi _{\delta }L(1-\pi _{\delta })N_{\xi }Lf_{\xi
}\Vert _{L^{1}}  \notag \\
& \leq \Vert N_{\xi }\pi _{\delta }L(1-\pi _{\delta })N_{\xi }L\tilde{f}%
\Vert _{L^{1}}+\Vert N_{\xi }\pi _{\delta }L(1-\pi _{\delta })N_{\xi
}L(f_{\xi }-\tilde{f})\Vert _{L^{1}}  \notag \\
& \leq \Vert N_{\xi }L(1-\pi _{\delta })N_{\xi }L\tilde{f}\Vert _{L^{1}} 
\notag \\
& \qquad +\Vert N_{\xi }(1-\pi _{\delta })L(1-\pi _{\delta })N_{\xi }L\tilde{%
f}\Vert _{L^{1}}  \notag \\
& \qquad +\Vert N_{\xi}\pi _{\delta }L\Vert _{L^{1}}\cdot \Vert (1-\pi
_{\delta })N_{\xi }\Vert _{L^{1}}\cdot \Vert f_{\xi }-\tilde{f}\Vert _{L^{1}}
\notag \\
& \leq \Vert N_{\xi }L(1-\pi _{\delta })N_{\xi }L\tilde{f}\Vert _{L^{1}} 
\notag \\
& \qquad +\Vert N_{\xi }(1-\pi _{\delta })\Vert _{L^{1}}\cdot \Vert 1-\pi
_{\delta }\Vert _{\Var\rightarrow L^{1}}\cdot \Var(N_{\xi }L\tilde{f}) 
\notag \\
& \qquad +\Vert (1-\pi _{\delta })N_{\xi }\Vert _{L^{1}}\cdot \Vert f_{\xi }-%
\tilde{f}\Vert _{L^{1}}  \notag \\
& \leq \Vert N_{\xi }L(1-\pi _{\delta })N_{\xi }L\tilde{f}\Vert_{L^{1}} \\
&\qquad +\frac{\delta ^{2}}{4}\Var(\rho _{\xi })\cdot \Var(N_{\xi }L\tilde{f}%
) \\
&\qquad +\frac{\delta }{2}\Var(\rho _{\xi })\cdot \Vert f_{\xi }-\tilde{f}%
\Vert _{L^{1}}.  \label{deltaq_est2}
\end{align}%
An algorithm for estimating for $\Var(N_{\xi }L\tilde{f})$ can be found in
Lemma \ref{Lemma6}. The term $\Vert N_{\xi }L(1-\pi _{\delta })N_{\xi }L%
\tilde{f}\Vert _{L^{1}}$ can be split over the intervals $I\in\Pi$ and
estimated as 
\begin{align*}
&\sum_{I\in\Pi}\Vert N_{\xi}L(1-\pi _{\delta })[N_{\xi }L\tilde{f}%
\cdot\chi_I]\Vert_{L^{1}} \\
&\qquad \leq \sum_{I\in\Pi}\min\bigg\{ \Vert N_{\xi}\Vert_{W\rightarrow
L^1}\cdot \Vert L\Vert_{W(I)\rightarrow W}\cdot \Vert 1-\pi _{\delta }\Vert_{%
\Var_I\rightarrow W(I)}, \\
&\qquad\qquad \Vert 1-\pi _{\delta }\Vert_{\Var_I\rightarrow L^1} \bigg\}%
\cdot\Var_I(N_{\xi}L\tilde{f}) \\
&\qquad \leq \sum_{I\in \Pi }\min \left\{ \frac{\delta ^{2}}{8}\Var(\rho
_{\xi })\cdot \Vert T^{\prime }\Vert _{L^{\infty}(I)},\frac{\delta }{2}%
\right\} \Var_{I}(N_{\xi }L\tilde{f})
\end{align*}
proving the statement thanks to Lemma \ref{LemmaW} (because each $I$ is a
union of intervals of the partition of size $\delta$).
\end{proof}

\begin{remark}
To estimate $B_3$, we estimate computationally $\mathrm{Var}_{I}(N_{\xi }L%
\tilde{f})$ for each interval $I\in\Pi$ using the algorithms explained in
Sections \ref{varLg} and \ref{varNLf}. As in Lemma \ref{lemma_N1pLf}, we
obtain a stronger estimate in the interval $I$ as soon as 
\begin{equation*}
\Vert T^{\prime }\Vert _{L^{\infty }(I)}<\frac{4}{\delta\xi^{-1} \mathrm{Var}%
(\rho)}.
\end{equation*}
Remark that in all our computations ${\delta\xi^{-1} \mathrm{Var}(\rho)}$ needs to be small, since it controls the approximation error (refer to Proposition \ref{summmary}, items 6 and 7). This implies that the inequality above will be true for most of the intervals but those where $T^{\prime}$ becomes very big.
\end{remark}

\subsubsection{An estimate for the $L^1$ error $\|f_\protect\xi -\tilde{f}%
\|_{L^1}$}

In the previous sections we built the ingredients for estimating $\Vert
f_{\xi }-f_{\xi ,\delta }\Vert _{L^{1}}$, but we want to estimate $\Vert
f_{\xi }-\tilde{f}\Vert _{L^{1}}$ where $\tilde{f}$ is the output of a
computation approximating the fixed point of $L_{\delta ,\xi }$. We will do
so assuming that we have an estimate of the numerical error $\Vert f_{\xi
,\delta }-\tilde{f}\Vert _{L^{1}}$ (such an estimate can be found in \cite%
{GN}).

Let $A_{i}$, $B_{i}$ ($i=1,2,3$) be the constants defined as in %
\eqref{ab1def}, \eqref{ab2def}, \eqref{ab3def}. Plugging \eqref{ab1}, %
\eqref{ab2}, \eqref{ab3} (according to Lemmas \ref{lemma_N1pLf} and \ref%
{lemma_NpL1pf}) into \eqref{aposteriori_telescopic}, we have that $\Vert
(L_{\delta ,\xi }^{\overline{n}}-L_{\xi }^{\overline{n}})f_{\xi }\Vert
_{L^{1}}$ can be bounded as 
\begin{equation*}
\Vert (L_{\delta ,\xi }^{\overline{n}}-L_{\xi }^{\overline{n}})f_{\xi }\Vert
_{L^{1}}\leq A\cdot \Vert f_{\xi }-\tilde{f}\Vert _{L^{1}}+B,
\end{equation*}%
where 
\begin{equation*}
A=A_{1}+(A_{2}+A_{3})\cdot \sum_{i=0}^{\overline{n}-1}C_{i}\qquad
B=B_{1}+(B_{2}+B_{3})\cdot \sum_{i=0}^{\overline{n}-1}C_{i}.
\end{equation*}%
Thanks to \eqref{eqmain} we have 
\begin{equation*}
\Vert f_{\xi }-f_{\xi ,\delta }\Vert _{L^{1}}\leq C+D\cdot \Vert f_{\xi }-%
\tilde{f}\Vert _{L^{1}}.
\end{equation*}%
for $C=A/(1-\alpha )$ and $D=B/(1-\alpha )$, $\alpha $ is appearing in %
\eqref{eqmain}. Therefore, by $($\ref{neweq1}$)$ we have  
\begin{equation*}
\Vert f_{\xi }-\tilde{f}\Vert _{L^{1}}\leq \Vert f_{\xi ,\delta }-\tilde{f}%
\Vert _{L^{1}}+C+D\cdot \Vert f_{\xi }-\tilde{f}\Vert _{L^{1}},
\end{equation*}%
which implies 
\begin{equation*}
\Vert f_{\xi }-\tilde{f}\Vert _{L^{1}}\leq \frac{1}{1-D}\cdot \left( \Vert
f_{\xi ,\delta }-\tilde{f}\Vert _{L^{1}}+C\right) .
\end{equation*}

\section{Contraction speed estimates via coarse-fine methods\label{coarsefine}}


In this section we show an efficient way to estimate the rate of contraction
of the discretized transfer operator and find the suitable $\overline{n}$
and $\alpha $ described in Section \ref{firstalgo}. Since $L_{\delta,\xi}$
is represented by a matrix, a first attempt to perform this task would be
to iterate and estimate the norm of the iterate. This method is not very
effective, since the matrix we should iterate is quite big. For this we
implement a strategy in which we get information for the full matrix from
the iterates of coarser versions of it. An earlier approach to this problem
for deterministic systems can be found in \cite{GNS}.

In Lemmas \ref{L1} and \ref{lem1mpLW} and Corollary \ref{c1} we have seen that 
\begin{equation}
\Vert (1-\pi _{\delta })N_{\xi }\Vert _{L^{1}\rightarrow L^{1}}\leq \delta
/\xi ,\qquad \Vert N_{\xi }(1-\pi _{\delta })\Vert _{L^{1}\rightarrow
L^{1}}\leq \delta /\xi ,  \label{necessest}
\end{equation}
We now prove the following lemma which bounds the distance between the powers
of $L_\xi$ and $L_{\delta,\xi}$, provided that the noise has been applied at
least once before the application of $L_\xi$ and $L_{\delta,\xi}$. 

\begin{lemma}
\label{cflemma} Let $\|L_{\delta,\xi}^i|_V\|_{L^1}\leq C_i$; let $\sigma$ be
a linear operator such that $\sigma^2=\sigma$, $||\sigma||_{L^1}\leq 1$, and 
$\sigma\pi_\delta=\pi_\delta\sigma=\pi_\delta$; let $\Lambda=\sigma
N_{\xi}\sigma L$.

Then $\forall n\geq 0$
\begin{equation}  \label{prima}
\|(L_{\delta,\xi}^n-\Lambda^n)N_\xi\|_{L^1} \leq \frac{\delta}{\xi}\cdot
\left(2\sum_{i=0}^{n-1}C_i+1\right)
\end{equation}
\end{lemma}

In particular the Lemma applies if:

\begin{enumerate}
\item $\sigma=Id$ and $\Lambda=L_{\xi}$;

\item $\sigma=\pi_{\delta^{\prime }}$ and $\Lambda=L_{\delta^{\prime },\xi}$
, for any $\delta^{\prime }$ such that $n\delta' =\delta$ with $n\in \mathbb{R}$.
\end{enumerate}

As a consequence, we obtain a way to bound the contraction rate of certain
operator $L_{\delta^{\prime },\xi}$ on the zero-average space $V$ using the
computed contraction rate for a coarser operator $L_{\delta,\xi}$. We remark
that 
\begin{equation}  \label{cfeq}
\| L_{\delta^{\prime },\xi} ^{n+1}|_V \|_{L^1} \leq \| L_{\delta,\xi}^n|_V
\|_{L^1} + \| (L_{\delta^{\prime },\xi}^n - L_{\delta,\xi}^n)\|_{L^1}.
\end{equation}
We remark that on the left-hand side we have an $n+1$ in \eqref{cfeq},
which guarantees that the noise has been applied at least once; this permits
us to use Lemma \ref{cflemma} to estimate the second summand of the
right-hand of \eqref{cfeq}.

Thus, if one is searching for an $n$ such that $|| L_{\delta^{\prime },\xi}
^{n+1}|_V \|_{L^1} <1 $ this can be found and certified by using a suitable
coarse version $L_{\delta,\xi}$, computing the norm of its iterates and
using Lemma \ref{cflemma} in a way that the second hand of (\ref{cfeq}) is
smaller than $1$.

\begin{proof}[Proof of Lemma \ref{cflemma}]
  Notice that as a consequence of the hypotheses we have
\[
||\sigma g-\pi_{\delta}g||_{L^1}\leq \Var(g)\delta/2, \quad ||\sigma g-\pi_{\delta}g||_{W}\leq ||g||_{L^1}\delta/2
\]
because $\sigma g-\pi_{\delta}g=\sigma(1-\pi_\delta)g$ applying Lemmas \ref{lem1mpLW} and
\ref{idmpVarW}.

The proof is along the lines of what has been proved in Section \ref{telescopicf}.
Indeed, we have
\begin{align*}
||(L_{\delta,\xi}^{n}-&\Lambda^{n})N_{\protect\xi}||_{L^1}\leq\\
&||(\pi_{\delta}N_{\xi}\pi_{\delta}L)^n||_{L^1}\cdot ||(\pi_{\delta}-\sigma)N_{\xi}||_{L^1}\\
&+\sum_{i=0}^{n-1}||(\pi_{\delta}N_{\xi}\pi_{\delta}L)^{i}||_{L^1}\cdot||(\pi_{\delta}-\sigma)N_{\xi}||_{L^1}\cdot|| (\sigma L\sigma N_{\xi})^{n-i}||_{L^1}\\
&+\sum_{i=0}^{n-1}||(\pi_{\delta}N_{\xi}\pi_{\delta}L)^{i}\pi_{\delta}||_{L^1}\cdot||N_{\xi}(\pi_{\delta}-\sigma)||_{L^1}\cdot||L\sigma N_{\xi}(\sigma L\sigma N_{\xi} )^{n-i-1}||_{L^1}\\
&\leq ||(\pi_{\delta}-\sigma)N_{\xi}||_{L^1}+\sum_{i=0}^{n-1}C_i ||(\pi_{\delta}-\sigma)N_{\xi}||_{L^1}\cdot|| (\sigma L\sigma N_{\xi})^{n-i}||_{L^1}\\
&+\sum_{i=0}^{n-1}C_i ||N_{\xi}(\pi_{\delta}-\sigma)||_{L^1}\cdot ||L\sigma N_{\xi}(\sigma L\sigma N_{\xi} )^{n-i-1}||_{L^1},
\end{align*}
and the thesis follows from the fact that $||L||_{L^1}\leq 1$, $||N_{\xi}||_{L^1}\leq 1$, $||\sigma||_{L^1}\leq 1$.
\end{proof}


\section{Estimating the average of an observable\label{Lyap}}


As a result of the previous sections, we are able to obtain a precise
approximation $\tilde{f}$ of $f_{\xi }$ in the $L^{1}$ norm. This is not
enough in order to estimate the Lyapunov exponent which we recall can be
defined as  $\lambda _{\xi }=\int h~df_{\xi }$ where  $h=\log |T^{\prime }|$%
. This is because $h$ is not in $L^{\infty }$ in the whole interval. In the
Belouzov-Zhabotinsky case, there are two points where $h$ goes to infinity:
the critical point and the point where the $|T^{\prime }|$ goes to $+\infty $%
. Outside of neighborhoods of these two points $h$ is bounded. Moreover, in
these two neighborhoods $h$ still has bounded $L^{1}$ norm. This is enough
to perform our estimates since we can have $L^{\infty }$ bounds on the
stationary measure $f_{\xi }$; this allows to compute the Lyapunov exponent 
using alternately  $L^{1}$ and $L^{\infty }$estimates on $f_{\xi }$ and $h$
in different sets. Therefore, we can join all these observations together to
obtain a rigorous approximation of $\lambda _{\xi }$  applying the following
strategy:

\begin{itemize}
\item we select a region $E$ of $[0,1]$ such that $h$ is in $L^{\infty }$
outside $E$;

\item we estimate, on $E$, the quantities $\Vert f_{\xi }\Vert _{L^{\infty
}(E)}$ and $\Vert h\Vert _{L^{1}(E)}$;

\item we approximate (keeping rigorously track of the numerical errors) the
integral $\int_{[0,1]}h~df_{\xi }$ with $\int_{[0,1]\setminus E}h~d\tilde{f}$
(discarding the set $E$ from the computation);

\item we estimate the error in such an approximation in terms of $\Vert
f_{\xi }-\tilde{f}\Vert _{L^{1}}$, $\Vert f_{\xi }\Vert _{L^{\infty }(E)}$, $%
\Vert h\Vert _{L^{1}(E)}$ and $\Vert h\Vert _{L^{\infty }([0,1]\setminus E)}$
(see Corollary \ref{corollsec5}).
\end{itemize}

In Subsection \ref{fistsubsection} we show how the error can be estimated in
terms of the mentioned quantities. The remaining subsections are devoted to
estimating $\Vert f_{\xi }\Vert _{L^{\infty }(E)}$ and $\Vert h\Vert
_{L^{1}(E)}$, notice that in the case of uniform noise $\Vert f_{\xi }\Vert
_{L^{\infty }(E)}\leq 1/\xi $, but we are able to obtain a better estimate
via $\tilde{f}$.



\subsection{Approximating the average using $L^{1}$ and $L^{\infty }$
estimates\label{fistsubsection}}

In this section we assume that $\tilde{f}$ is an approximation of $f_\xi$
and both are probability measures, therefore $\tilde{f}-f_\xi$ has $0$
average.
\begin{corollary}
\label{corollsec5}Let $f$ and $\tilde{f}$ be probability densities on the
measure space $(X,m)$, both contained in $L^{1}$ and in $L^{\infty }$.
Let $E\subset X$ be a Borel subset, and $H$ be an $L^{1}$ observable that is 
$L^{\infty }$ in $X\setminus E$. Then 
\begin{equation*}
\scalebox{0.95}{$\left|\int_X Hf \ dm - \int_{X\setminus E} H\tilde{f} \ dm
\right| \leq \|H\|_{L^1(E)}\cdot\|f\|_{L^\infty(E)} + \frac{\sup_{X\setminus
E} H + \inf H_{X\setminus E}}{2} \cdot \|f-\tilde{f}\|_{L^1}.$}
\end{equation*}
\end{corollary}

The proof of the corollary is straightforward, applying the following Lemma on
the set $X\setminus E$

\begin{lemma}
Let $(X,m)$ be a measure space, let $H\in L^\infty(X)$, and let $v\in L^1(X)$
a function having $0$ average. Then we have 
\begin{equation*}
\left| \int_X H\cdot v \ dm \right| \leq \frac{\sup H -\inf H}{2} \cdot
\|v\|_{L^1}.
\end{equation*}
\end{lemma}

\begin{proof}
Indeed, for a constant $c$ we have 
\begin{align*}
\left| \int_X H v \ dm \right| &\leq \left| \int_X (H-c) v \ dm\right| +
\left| \int_X c v \ dm \right| \\
&\leq \|H-c\|_{L^\infty} \cdot \|v\|_{L^1},
\end{align*}
because $v$ has $0$ average, and this is clearly optimized taking $c=(\sup H
+ \inf H)/2$.
\end{proof}

\begin{remark}
\label{obsestalg} Corollary \ref{corollsec5} yields immediately an algorithm for
estimating an  observable that is $L^1$, and is $L^\infty$ outside a
neighborhood of a finite number of  points $s_i$ where it goes to $\infty$,
as is the case with the observable  $\log|T^{\prime }|$ of the system we are
studying. In fact, there is a trade-off on the  size of $E$, and we attempt
different sets $E$ enclosing the $ s_i$ with intervals of different sizes on
order to obtain the tightest possible estimate  on the error. Every such
choice of the set $E$ yields an approximation of  $\int_X Hf \ dm $ as $%
\int_{X\setminus E} H\tilde{f} \ dm $ and a bound for the error.
\end{remark}

\subsection{$L^\infty$ bounds for the stationary measure in an interval.}

\label{LocLInfty}

To estimate the average of the unbonded observable $h$ and apply Corollary %
\ref{corollsec5}, in this Subsection we obtain a bound for the $L^{\infty }$
norm of the invariant measure $f_{\xi }$ on intervals of a certain
partition. \ We derive it as a byproduct of the rigorous estimate of the $%
L^{1}$ error, and the algorithms explained in Sections \ref{varLg} and \ref%
{varNLf}, that allow to bound $\mathrm{Var}_{I}(N_{\xi }L\tilde{f})$ for
each $I\in \Pi $.

\begin{lemma}
Let $\Pi$ be a uniform partition, for each $I\in\Pi$ we have 
\begin{equation*}
\|f_{\xi}\|_{L^\infty(I)} \leq \mathrm{Var}_I(N_\xi L\tilde{f}) + \frac{%
\|N_\xi L\tilde{f}\|_{L^1(I)}}{|I|} + \|\tilde{f}-f_{\xi}\|_{L^1}\cdot
\|\rho_\xi\|_{L^\infty}.
\end{equation*}
\end{lemma}

\begin{proof}
Indeed, 
\begin{align*}
\|f_{\xi}\|_{L^\infty(I)} &= \|N_\xi Lf_{\xi}\|_{L^\infty(I)} \\
&\leq \|N_\xi L\tilde{f}\|_{L^\infty(I)} + \|N_\xi L(\tilde{f}%
-f_{\xi})\|_{L^\infty} \\
&\leq \Var_I(N_\xi L\tilde{f}) + \|N_\xi L\tilde{f}\|_{L^1(I)}/ |I| +
\|N_\xi\|_{L^1\rightarrow L^\infty} \cdot \|L(\tilde{f}-f_{\xi})\|_{L^1}
\end{align*}
and the estimate follows because $\|N_\xi\|_{L^1\rightarrow L^\infty}\leq
\|\rho_\xi\|_{L^\infty}$.
\end{proof}

\subsection{$L^1$ bounds on $\log|T^{\prime }|$}

In this section we compute explicit bounds for the $L^1$-norm of $%
\log|T^{\prime }|$ for the map defined in Section \ref{map} in a
neighborhood of the points where it is not bounded, as required to apply
Remark \ref{obsestalg}. As can be deduced from its definition in Section \ref%
{map}, we need to do so in intervals enclosing $x=0.125$ and $x=0.3$. We
omit the proofs, as they are all very elementary.

\begin{lemma}
For $0.125-2^{-6}<u<0.125<v<0.125+2^{-6}$ we have 
\begin{align*}
\int_u^v \log|T^{\prime }(x)| \ dx \in &-\frac{2}{3}\bigg[%
(0.125-u)(\log(0.125-u)-1)+ (v-0.125)(\log(v-0.125)-1)\bigg] \\
&- (v-u)\log(3) - \frac{v^2-u^2}{2} + [0, (\log(5)-\log(4))(v-u)]
\end{align*}
\end{lemma}

\begin{lemma}
For $0.2<x<0.3$ we have 
\begin{equation*}
\int_{x}^{0.3} \log|T^{\prime }(x)| \ dx \in
\left(\log([d_1,d_2])+\log(0.3-x)-1\right)(0.3-x) - \frac{1}{2}(0.3^2-x^2).
\end{equation*}
for $d_1 = \frac{1}{3}(\frac{2}{3}\cdot (0.175)^{-5/3} + (0.175)^{-2/3})$
and $d_2 =[a + (x-0.125)^{1/3} - \frac{1}{3}|x-0.125|^{-2/3}]/(0.3-x)$,
where $a=0.50607356\dots$ as in Section \ref{map}.
\end{lemma}

\begin{lemma}
For $0.3 < x < 0.303$ we have 
\begin{align*}
\int_{0.3}^{x}\log |T^{\prime }(x)| \ dx= & \bigg(\log k_1 + \log 19 + 19
\log 10 - 38 + \log(10/3)\bigg)\cdot (x-0.3) + \\
& 18(x\log(x)+0.3\log(0.3))+(x-0.3)\log(x-0.3)-\frac{190}{6}(x-0.3)^2.
\end{align*}
\end{lemma}

\section{Computation details and results}\label{results}


\label{compresults}



\begin{figure}[tbp]
\begin{subfigure}[b]{0.45\textwidth}
   \includegraphics[width=55mm,height=50mm]{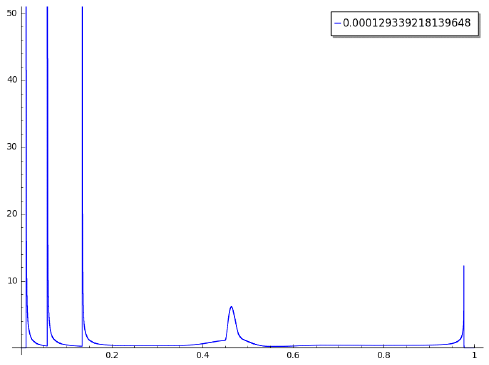} 
   \caption{A plot of of the approximated invariant density for $\xi_1=0.129 \times 10^{-3}$ (up to an error of $0.565\times 10^{-3}$ in the $L^1$ norm).}
  \end{subfigure}  
\begin{subfigure}[b]{0.45\textwidth}
    \includegraphics[width=55mm,height=50mm]{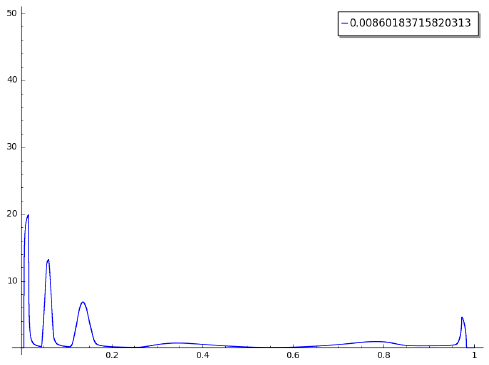}
   \caption{A plot of of the approximated invariant density for $\xi_2=0.860\times 10^{-2}$ (up to an error of $0.973\times 10^{-5}$ in the $L^1$ norm)}
  \end{subfigure}
\caption{Plot of the computed density for the noise amplitudes in
Proposition \protect\ref{theo:final}.}
\label{fig:densities}
\end{figure}

We give here some details about the code performing our computer aided
estimates and about the results. The main algorithm is written in Python
using the Sage framework \ and interval arithmetics (\cite{Tucker}), some
critical parts are written in C++ and uses (optionally) the GPU, this in
particular is used for the iterartion of large matrices needed to apply the
methods of Section \ref{coarsefine}. Such parts have been optimized to use
high performance computing; even so each contraction test has required a
time of the of order of a week. \footnote{The contraction time estimates were run on a Asus GeForce GTX 1050Ti, 4GB of Ram GPU installed in   a desktop computer with an AMD A4-6300 3.4 Ghz processor and 8 Gb of Ram.
The matrices were assembled on a Dell R710 server with 2 sixcore Xeon 5660 2.8 Ghz processors and 24 Gb of Ram.}
 The code we used can be found at
\begin{verbatim}
 http://im.ufrj.br/~maurizio.monge/wordpress/rigorous_computation_dyn/
 
\end{verbatim}

%
%

\begin{figure}[htbp]
\centering
\includegraphics[width=4in]{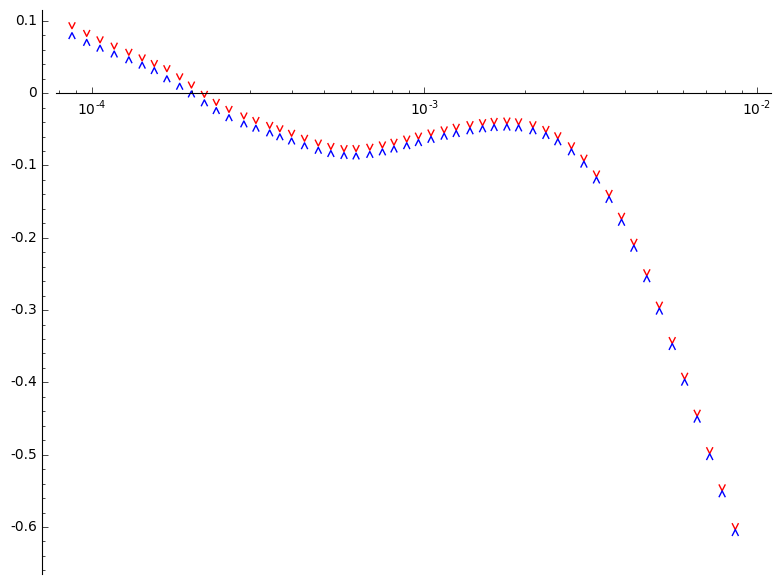}
\caption{A plot of the intervals enclosing the values of the Lyapunov
exponent for several sizes of the noise. The plotted values are listed in
Table 1.}
\label{flyap}
\end{figure}

Table \ref{t1} contains the result of the computer aided estimates we
performed and the values of the parameters used in these estimates. In
Figure \ref{flyap} we summarize with a graph the most important information
contained in the table. The graph shows intervals enclosing the Lyapunov
exponent at the selected noise values. These are the final result of our
computer aided estimates; it is worth to remark again that the estimated
requiring more computational power are the ones performed to prove that the
Lyapunov exponent is positive for small size of the noise.

To ease the understanding of Table \ref{t1}, we now outline some more
details of the implementation of our algorithm and how the parameters take
place in the algorithm's execution; the columns are ordered as they are
subsequently used in the algorithm, or deduced from previous quantities and
via computations.

One of the main ingredients and goals of the computer aided estimates is to
compute the number of iterates of the transfer operator needed to contract
the zero average space (see Item 1 of Section \ref{infoalgo}). For this we
apply the "coarse fine" methods explained in Section \ref{coarsefine}. For
each value of the noise amplitude, denoted by $\xi $, we build a coarse
discretization $L_{\xi ,\delta _{\text{contr}}}$ of the operator $L_{\xi }$,
on a partition of coarse size $\delta _{\text{contr}}$. We then compute
values of $n_{\text{contr}}$ and $\alpha _{\text{contr}}$ represented in
Table \ref{t1}, which satisfy 
\begin{equation*}
||L_{\xi ,\delta _{\text{contr}}}^{n_{\text{contr}}}|_{V}||_{L^{1}}\leq
\alpha _{\text{contr}}<1;
\end{equation*}%
and compute explicit bounds for $||L_{\xi ,\delta _{\text{contr}%
}}^{i}|_{V}||_{L^{1}}$ (whose value does not appear in the table). The
algorithm used for these finite dimensional estimates is the same as in \cite%
{GN} and there explained.

We consider then a finer partition size $\delta $ and use the coarse fine
estimates of Section \ref{coarsefine} to compute $\alpha $ and $%
\sum_{i=0}^{n_{\text{contr}}}C_{i}$, where $C_{i}=||L_{\xi ,\delta
}^{i}|_{V}||_{L^{1}}$ and
\begin{equation*}
||L_{\xi ,\delta }^{n_{\text{contr}}+1}|_{V}||_{L^{1}}\leq \alpha .
\end{equation*}

The same bounds work for $L_{\xi }$, i.e. 
\begin{equation*}
||L_{\xi }^{n_{\text{contr}}+1}|_{V}||_{L^{1}}\leq \alpha
,~and~\sum_{i=0}^{n_{\text{contr}}}||L_{\xi }^{i}|_{V}||_{L^{1}}\leq
\sum_{i=0}^{n_{\text{contr}}}C_{i}.
\end{equation*}%
Notice that the bigger the ${n_{\text{contr}}}$, the worse is going to be
the estimate on the $L^{1}$ norm on $(L_{\xi }^{n}-L_{\xi ,\delta _{\text{%
contr}}}^{n})N_{\xi }$, i.e., the error coming from the coarse-fine
inequality. While increasing $n_{\text{contr}}$ permits us to find smaller $%
\alpha _{\text{contr}}$, this may not imply that the corresponding $\alpha $
is smaller. Our algorithms attempts to find the best compromise. Needless to
say, in practice this procedure may fail, if unable to detect any
contraction in a reasonable time. This might happen for example if the original system is not mixing.

To estimate an upper bound to the $L^{1}$ error in the computation of the
stationary measure we use the results shown in Section \ref{sec3}. The
column "A priori..." contains the a priori estimate on the $L^{1}$ error of
the approximation of the measure on the partition of size $\delta $ as given
using the results in Subsection \ref{init} while the column "Refined..."
contains the $L^{1}$ error when we use the bootstrapping tecniques of
Subsection \ref{post}.

Once we have a good approximation of the invariant measure, we compute an
approximation of the Lyapunov exponent by computing an integral on a
partition of size $\delta_{\text{est}}$, using Section \ref{Lyap}; the
computed intervals enclosing rigorously the Lyapunov exponent are contained
in the last column. This allows to prove Theorem \ref{theo:final}.

\begin{table}[htbp]
\centering
\renewcommand{\arraystretch}{1.3} 
\scalebox{0.6}{
\begin{tabular}{c|c|c|c|c|c|c|c|c|c|c}
  \raisebox{5pt}{$\textrm{$\xi$, noise size}$} & \raisebox{5pt}{$\delta_{\textrm{contr}}$} &
  \raisebox{5pt}{$\alpha_{\textrm{contr}}$} & \raisebox{5pt}{$n_{\textrm{contr}}$} &
  \raisebox{5pt}{$\delta$} & \raisebox{5pt}{$\alpha$} &
  \raisebox{5pt}{$\sum_{i=0}^{n_{\text{contr}}} C_i$} & \shortstack{\emph{a priori} $L^1$ err.\\on measure} &
  \raisebox{5pt}{$\delta_{\textrm{est}}$} & \shortstack{refined $L^1$ err.\\on measure} &
  \shortstack{rig. estimate on\\Lypunov exponent} \\
 \hline
$0.860\times 10^{-2} $ & $2^{-18}$ & $0.023$ & $56$ & $2^{-27}$ & $0.044$ & $29.54$ & $0.445\times 10^{-4} $ & $2^{-14}$ & $0.246\times 10^{-5} $ & $-6.03^{536}_{602}\times 10^{-1} $\\ 
$0.785\times 10^{-2} $ & $2^{-18}$ & $0.017$ & $58$ & $2^{-27}$ & $0.039$ & $29.20$ & $0.479\times 10^{-4} $ & $2^{-14}$ & $0.257\times 10^{-5} $ & $-5.499^{221}_{901}\times 10^{-1} $\\ 
$0.721\times 10^{-2} $ & $2^{-18}$ & $0.013$ & $60$ & $2^{-27}$ & $0.037$ & $28.95$ & $0.517\times 10^{-4} $ & $2^{-14}$ & $0.268\times 10^{-5} $ & $-4.985^{026}_{749}\times 10^{-1} $\\ 
$0.661\times 10^{-2} $ & $2^{-18}$ & $0.015$ & $60$ & $2^{-27}$ & $0.04$ & $28.75$ & $0.562\times 10^{-4} $ & $2^{-14}$ & $0.282\times 10^{-5} $ & $-4.467^{104}_{862}\times 10^{-1} $\\ 
$0.606\times 10^{-2} $ & $2^{-18}$ & $0.022$ & $54$ & $2^{-27}$ & $0.05$ & $28.44$ & $0.612\times 10^{-4} $ & $2^{-14}$ & $0.296\times 10^{-5} $ & $-3.95^{558}_{636}\times 10^{-1} $\\ 
$0.556\times 10^{-2} $ & $2^{-18}$ & $0.015$ & $58$ & $2^{-27}$ & $0.046$ & $28.77$ & $0.672\times 10^{-4} $ & $2^{-14}$ & $0.311\times 10^{-5} $ & $-3.462^{063}_{884}\times 10^{-1} $\\ 
$0.509\times 10^{-2} $ & $2^{-18}$ & $0.014$ & $60$ & $2^{-27}$ & $0.048$ & $29.19$ & $0.747\times 10^{-4} $ & $2^{-14}$ & $0.330\times 10^{-5} $ & $-2.97^{272}_{360}\times 10^{-1} $\\ 
$0.466\times 10^{-2} $ & $2^{-18}$ & $0.019$ & $59$ & $2^{-27}$ & $0.056$ & $29.56$ & $0.832\times 10^{-4} $ & $2^{-14}$ & $0.351\times 10^{-5} $ & $-2.52^{476}_{568}\times 10^{-1} $\\ 
$0.426\times 10^{-2} $ & $2^{-18}$ & $0.032$ & $53$ & $2^{-27}$ & $0.072$ & $29.70$ & $0.930\times 10^{-4} $ & $2^{-14}$ & $0.374\times 10^{-5} $ & $-2.10^{322}_{419}\times 10^{-1} $\\ 
$0.391\times 10^{-2} $ & $2^{-18}$ & $0.025$ & $57$ & $2^{-27}$ & $0.07$ & $30.68$ & $0.104\times 10^{-3} $ & $2^{-14}$ & $0.400\times 10^{-5} $ & $-1.74^{370}_{474}\times 10^{-1} $\\ 
$0.359\times 10^{-2} $ & $2^{-18}$ & $0.023$ & $60$ & $2^{-27}$ & $0.074$ & $31.73$ & $0.118\times 10^{-3} $ & $2^{-14}$ & $0.431\times 10^{-5} $ & $-1.42^{733}_{844}\times 10^{-1} $\\ 
$0.329\times 10^{-2} $ & $2^{-18}$ & $0.028$ & $60$ & $2^{-27}$ & $0.085$ & $32.63$ & $0.134\times 10^{-3} $ & $2^{-14}$ & $0.467\times 10^{-5} $ & $-1.1^{588}_{600}\times 10^{-1} $\\ 
$0.302\times 10^{-2} $ & $2^{-18}$ & $0.034$ & $60$ & $2^{-27}$ & $0.097$ & $33.55$ & $0.152\times 10^{-3} $ & $2^{-14}$ & $0.510\times 10^{-5} $ & $-9.^{399}_{411}\times 10^{-2} $\\ 
$0.277\times 10^{-2} $ & $2^{-18}$ & $0.044$ & $58$ & $2^{-27}$ & $0.11$ & $34.26$ & $0.173\times 10^{-3} $ & $2^{-14}$ & $0.560\times 10^{-5} $ & $-7.^{692}_{706}\times 10^{-2} $\\ 
$0.252\times 10^{-2} $ & $2^{-18}$ & $0.053$ & $56$ & $2^{-27}$ & $0.13$ & $35.04$ & $0.198\times 10^{-3} $ & $2^{-14}$ & $0.626\times 10^{-5} $ & $-6.3^{035}_{187}\times 10^{-2} $\\ 
$0.232\times 10^{-2} $ & $2^{-18}$ & $0.053$ & $56$ & $2^{-27}$ & $0.14$ & $35.94$ & $0.223\times 10^{-3} $ & $2^{-14}$ & $0.693\times 10^{-5} $ & $-5.4^{305}_{472}\times 10^{-2} $\\ 
$0.212\times 10^{-2} $ & $2^{-18}$ & $0.055$ & $56$ & $2^{-27}$ & $0.15$ & $36.88$ & $0.254\times 10^{-3} $ & $2^{-14}$ & $0.784\times 10^{-5} $ & $-4.7^{769}_{957}\times 10^{-2} $\\ 
$0.192\times 10^{-2} $ & $2^{-18}$ & $0.051$ & $58$ & $2^{-27}$ & $0.16$ & $38.21$ & $0.293\times 10^{-3} $ & $2^{-14}$ & $0.909\times 10^{-5} $ & $-4.3^{776}_{993}\times 10^{-2} $\\ 
$0.177\times 10^{-2} $ & $2^{-18}$ & $0.053$ & $58$ & $2^{-27}$ & $0.18$ & $38.98$ & $0.329\times 10^{-3} $ & $2^{-14}$ & $0.104\times 10^{-4} $ & $-4.2^{654}_{897}\times 10^{-2} $\\ 
$0.162\times 10^{-2} $ & $2^{-18}$ & $0.054$ & $58$ & $2^{-27}$ & $0.19$ & $39.82$ & $0.373\times 10^{-3} $ & $2^{-14}$ & $0.121\times 10^{-4} $ & $-4.3^{194}_{479}\times 10^{-2} $\\ 
$0.150\times 10^{-2} $ & $2^{-18}$ & $0.055$ & $58$ & $2^{-27}$ & $0.2$ & $40.53$ & $0.419\times 10^{-3} $ & $2^{-14}$ & $0.139\times 10^{-4} $ & $-4.^{488}_{521}\times 10^{-2} $\\ 
$0.137\times 10^{-2} $ & $2^{-18}$ & $0.063$ & $57$ & $2^{-27}$ & $0.23$ & $41.00$ & $0.476\times 10^{-3} $ & $2^{-14}$ & $0.164\times 10^{-4} $ & $-4.7^{613}_{992}\times 10^{-2} $\\ 
$0.125\times 10^{-2} $ & $2^{-18}$ & $0.066$ & $58$ & $2^{-27}$ & $0.25$ & $42.14$ & $0.554\times 10^{-3} $ & $2^{-14}$ & $0.200\times 10^{-4} $ & $-5.1^{293}_{750}\times 10^{-2} $\\ 
$0.115\times 10^{-2} $ & $2^{-18}$ & $0.071$ & $58$ & $2^{-27}$ & $0.27$ & $43.12$ & $0.636\times 10^{-3} $ & $2^{-14}$ & $0.239\times 10^{-4} $ & $-5.^{491}_{545}\times 10^{-2} $\\ 
$0.105\times 10^{-2} $ & $2^{-18}$ & $0.079$ & $58$ & $2^{-27}$ & $0.3$ & $44.28$ & $0.747\times 10^{-3} $ & $2^{-14}$ & $0.294\times 10^{-4} $ & $-5.9^{172}_{835}\times 10^{-2} $\\ 
$0.960\times 10^{-3} $ & $2^{-18}$ & $0.086$ & $58$ & $2^{-27}$ & $0.34$ & $45.41$ & $0.876\times 10^{-3} $ & $2^{-14}$ & $0.360\times 10^{-4} $ & $-6.^{346}_{427}\times 10^{-2} $\\ 
$0.885\times 10^{-3} $ & $2^{-18}$ & $0.092$ & $58$ & $2^{-27}$ & $0.37$ & $46.43$ & $0.102\times 10^{-2} $ & $2^{-14}$ & $0.436\times 10^{-4} $ & $-6.^{758}_{855}\times 10^{-2} $\\ 
$0.810\times 10^{-3} $ & $2^{-18}$ & $0.095$ & $58$ & $2^{-27}$ & $0.4$ & $47.44$ & $0.119\times 10^{-2} $ & $2^{-14}$ & $0.534\times 10^{-4} $ & $-7.^{200}_{320}\times 10^{-2} $\\ 
$0.748\times 10^{-3} $ & $2^{-18}$ & $0.098$ & $58$ & $2^{-27}$ & $0.43$ & $48.25$ & $0.138\times 10^{-2} $ & $2^{-14}$ & $0.643\times 10^{-4} $ & $-7.^{569}_{711}\times 10^{-2} $\\ 
$0.686\times 10^{-3} $ & $2^{-18}$ & $0.1$ & $57$ & $2^{-27}$ & $0.46$ & $48.65$ & $0.162\times 10^{-2} $ & $2^{-14}$ & $0.779\times 10^{-4} $ & $-^{7.888}_{8.061}\times 10^{-2} $\\ 
$0.623\times 10^{-3} $ & $2^{-18}$ & $0.099$ & $56$ & $2^{-27}$ & $0.5$ & $49.04$ & $0.192\times 10^{-2} $ & $2^{-14}$ & $0.956\times 10^{-4} $ & $-8.^{068}_{280}\times 10^{-2} $\\ 
$0.573\times 10^{-3} $ & $2^{-19}$ & $0.07$ & $60$ & $2^{-27}$ & $0.29$ & $46.78$ & $0.141\times 10^{-2} $ & $2^{-15}$ & $0.595\times 10^{-4} $ & $-8.^{076}_{209}\times 10^{-2} $\\ 
$0.524\times 10^{-3} $ & $2^{-19}$ & $0.085$ & $56$ & $2^{-27}$ & $0.33$ & $46.32$ & $0.162\times 10^{-2} $ & $2^{-15}$ & $0.696\times 10^{-4} $ & $-7.^{784}_{940}\times 10^{-2} $\\ 
$0.480\times 10^{-3} $ & $2^{-19}$ & $0.07$ & $58$ & $2^{-27}$ & $0.34$ & $47.75$ & $0.185\times 10^{-2} $ & $2^{-15}$ & $0.800\times 10^{-4} $ & $-7.^{312}_{490}\times 10^{-2} $\\ 
$0.436\times 10^{-3} $ & $2^{-19}$ & $0.066$ & $58$ & $2^{-27}$ & $0.36$ & $48.66$ & $0.215\times 10^{-2} $ & $2^{-15}$ & $0.935\times 10^{-4} $ & $-6.^{653}_{861}\times 10^{-2} $\\ 
$0.399\times 10^{-3} $ & $2^{-19}$ & $0.06$ & $60$ & $2^{-27}$ & $0.39$ & $50.39$ & $0.254\times 10^{-2} $ & $2^{-15}$ & $0.110\times 10^{-3} $ & $-^{5.964}_{6.208}\times 10^{-2} $\\ 
$0.368\times 10^{-3} $ & $2^{-19}$ & $0.065$ & $60$ & $2^{-27}$ & $0.43$ & $51.36$ & $0.299\times 10^{-2} $ & $2^{-15}$ & $0.128\times 10^{-3} $ & $-5.^{339}_{623}\times 10^{-2} $\\ 
$0.343\times 10^{-3} $ & $2^{-20}$ & $0.041$ & $65$ & $2^{-27}$ & $0.24$ & $49.27$ & $0.232\times 10^{-2} $ & $2^{-16}$ & $0.976\times 10^{-4} $ & $-^{4.871}_{5.086}\times 10^{-2} $\\ 
$0.312\times 10^{-3} $ & $2^{-20}$ & $0.044$ & $65$ & $2^{-27}$ & $0.26$ & $50.33$ & $0.270\times 10^{-2} $ & $2^{-16}$ & $0.112\times 10^{-3} $ & $-4.^{189}_{436}\times 10^{-2} $\\ 
$0.287\times 10^{-3} $ & $2^{-20}$ & $0.048$ & $65$ & $2^{-27}$ & $0.29$ & $51.29$ & $0.309\times 10^{-2} $ & $2^{-16}$ & $0.127\times 10^{-3} $ & $-3.^{565}_{844}\times 10^{-2} $\\ 
$0.259\times 10^{-3} $ & $2^{-20}$ & $0.042$ & $67$ & $2^{-27}$ & $0.31$ & $53.14$ & $0.368\times 10^{-2} $ & $2^{-16}$ & $0.149\times 10^{-3} $ & $-2.^{659}_{984}\times 10^{-2} $\\ 
$0.237\times 10^{-3} $ & $2^{-20}$ & $0.046$ & $67$ & $2^{-27}$ & $0.35$ & $54.32$ & $0.431\times 10^{-2} $ & $2^{-17}$ & $0.144\times 10^{-3} $ & $-1.^{686}_{998}\times 10^{-2} $\\ 
$0.218\times 10^{-3} $ & $2^{-20}$ & $0.05$ & $67$ & $2^{-27}$ & $0.38$ & $55.49$ & $0.504\times 10^{-2} $ & $2^{-17}$ & $0.165\times 10^{-3} $ & $-^{5.731}_{9.274}\times 10^{-3} $\\ 
$0.199\times 10^{-3} $ & $2^{-20}$ & $0.051$ & $68$ & $2^{-27}$ & $0.42$ & $57.28$ & $0.606\times 10^{-2} $ & $2^{-17}$ & $0.193\times 10^{-3} $ & $^{7.138}_{2.988}\times 10^{-3} $\\ 
$0.184\times 10^{-3} $ & $2^{-20}$ & $0.051$ & $69$ & $2^{-27}$ & $0.46$ & $59.07$ & $0.725\times 10^{-2} $ & $2^{-17}$ & $0.225\times 10^{-3} $ & $1.^{844}_{362}\times 10^{-2} $\\ 
$0.168\times 10^{-3} $ & $2^{-20}$ & $0.053$ & $70$ & $2^{-27}$ & $0.5$ & $61.10$ & $0.896\times 10^{-2} $ & $2^{-17}$ & $0.272\times 10^{-3} $ & $2.^{977}_{394}\times 10^{-2} $\\ 
$0.154\times 10^{-3} $ & $2^{-21}$ & $0.042$ & $74$ & $2^{-27}$ & $0.29$ & $57.78$ & $0.650\times 10^{-2} $ & $2^{-18}$ & $0.708\times 10^{-4} $ & $3.^{689}_{533}\times 10^{-2} $\\ 
$0.142\times 10^{-3} $ & $2^{-21}$ & $0.048$ & $74$ & $2^{-27}$ & $0.32$ & $59.11$ & $0.757\times 10^{-2} $ & $2^{-18}$ & $0.822\times 10^{-4} $ & $4.^{462}_{282}\times 10^{-2} $\\ 
$0.129\times 10^{-3} $ & $2^{-21}$ & $0.049$ & $75$ & $2^{-27}$ & $0.36$ & $60.98$ & $0.901\times 10^{-2} $ & $2^{-18}$ & $0.988\times 10^{-4} $ & $5.^{239}_{023}\times 10^{-2} $\\ 
$0.106\times 10^{-3} $ & $2^{-21}$ & $0.058$ & $75$ & $2^{-27}$ & $0.45$ & $64.34$ & $0.135\times 10^{-1} $ & $2^{-18}$ & $0.154\times 10^{-3} $ & $6.^{965}_{626}\times 10^{-2} $\\ 
$0.873\times 10^{-4} $ & $2^{-21}$ & $0.062$ & $75$ & $2^{-27}$ & $0.55$ & $67.55$ & $0.209\times 10^{-1} $ & $2^{-18}$ & $0.252\times 10^{-3} $ & $8.^{917}_{365}\times 10^{-2} $\\ 

\end{tabular}
} \renewcommand{\arraystretch}{1}
\caption{The input and output of our computer aided estimates. In particular
the first column shows the size of the noise, the fifth column shows the
size of the approximation grid, the last column shows intervals enclosing
the exact value for the Lyapunov exponent related to the noise size in the
first column. See Section \protect\ref{compresults} for explanations on all
the other values.}
\label{t1}
\end{table}

\begin{proof}[Proof of Items I1 and I2 of Theorem \protect\ref{theo:final}]
We refer to the values listed in Table \ref{t1}. For a noise size of $\xi
_{1}=0.873\times 10^{-4}$. By the results of Sections 3,4,5 our algorithm
certifies that the Lyapunov exponent $\lambda _{\xi _{1}}\in \lbrack
8.365\times 10^{-2},8.917\times 10^{-2}]$. In particular, this proves that $%
\lambda _{\xi _{1}}>0$.

For a noise size of $\xi_2=0.860\times 10^{-2}$ our algorithm certifies that
the Lyapunov exponent $\lambda_{\xi_2}\in [-6.03602\times
10^{-1},-6.03536\times 10^{-1}]$. In particular, this proves that $%
\lambda_{\xi_2}<0$.
\end{proof}

We remark that in the case of random diffeomorphisms there exists a
dychotomy \cite{LY-Lyapunov}: if the Lyapunov exponent is positive the
system admits a \textit{random strange attractor} (chaotic behaviour) while
if the Lyapunov exponent is negative the system has a \textit{random sink}
(regular behaviour).



\section{Quantitative stability of the system and of the Lyapunov exponent.}\label{sec7}


In this section we study of the regularity of $\lambda _{\xi }$ as a
function of $\xi $ and prove that this varies $\alpha -$Holder  continuously
for every $\alpha <1$. We start showing a simple Lipschitz stability result
for the fixed point of a Markov operator. This will show that the stationary
measure is Lipschitz stable in $L^{1}$ when the noise amplitude change. This
is not sufficient to deduce that the Lyapunov exponent is Lipschitz stable,
because as we have seen the Lyapunov exponent is the average of an
observable which is unbounded. In this case a control in $L^{1}$ of the
stationary measure is not sufficient to control its average. For this reason
we strengthen the estimates to an  $L^{p}$ quantitative stability statement
with $p>1$ which is sufficient to control the average of our observable.

\subsection{Lipschitz stability of the stationary measure}

A quantitative stability statement for the stationary measure follows from a
general and elementary lemma about perturbations of Markov operators:

\begin{lemma}
\label{gen.}Let $L_{1},L_{2}:L^{1}\rightarrow L^{1}$ be two Markov
operators. Assume 
\begin{equation*}
\Vert L_{1}^{i}|_{V}\Vert _{L^{1}\rightarrow L^{1}}\leq C_{i}
\end{equation*}%
and suppose that $C_{N}<1$ for a certain $N$. Let $f_{i}$ be a fixed
probability measure of $L_{i}$, then $\sum_{k=0}^{N-1}C_{i}<\infty $ and 
\begin{equation}
\Vert f_{2}-f_{1}\Vert _{L^{1}}\leq \frac{\sum_{k=0}^{N-1}C_{i}}{1-C_{N}}%
\cdot \Vert L_{1}-L_{2}\Vert _{L^{1}}.  \label{eqstab}
\end{equation}
\end{lemma}

\begin{proof}
The existence of $N$such that $C_{N}<1$ easily implies that $C_{i}$
decreases exponentially and $\sum_{k=0}^{N-1}C_{i}<\infty $. Since $%
f_{0},f_{1}$ are fixed probability measures 
\begin{eqnarray*}
\Vert f_{2}-f_{1}\Vert _{L^{1}} &\leq &\Vert
L_{2}^{N}f_{2}-L_{1}^{N}f_{1}\Vert _{L^{1}} \\
&\leq &\Vert L_{2}^{N}f_{2}-L_{1}^{N}f_{2}\Vert _{L^{1}}+\Vert
L_{1}^{N}f_{2}-L_{1}^{N}f_{1}\Vert _{L^{1}} \\
&\leq &\Vert L_{1}^{N}(f_{2}-f_{1})\Vert _{L^{1}}+\Vert
L_{2}^{N}f_{2}-L_{1}^{N}f_{2}\Vert _{L^{1}}.
\end{eqnarray*}%
Since $C_{i}\rightarrow 0$, and $f_{2}-f_{1}\in V$, we can chose $N$ such
that $C_{N}<1$, we have $\Vert L_{1}^{N}(f_{2}-f_{1})\Vert _{L^{1}}\leq
C_{N}\Vert f_{2}-f_{1}\Vert _{L^{1}}$ and 
\begin{equation*}
\Vert f_{2}-f_{1}\Vert _{L^{1}}\leq \frac{\Vert
L_{2}^{N}f_{2}-L_{1}^{N}f_{2}\Vert _{L^{1}}}{1-C_{N}}.
\end{equation*}%
Let us now consider the term $\Vert L_{2}^{N}f_{2}-L_{1}^{N}f_{2}\Vert
_{L^{1}}$. Since%
\begin{equation*}
(L_{1}^{N}-L_{2}^{N})=\sum_{k=1}^{N}L_{1}^{N-k}(L_{1}-L_{2})L_{2}^{k-1}
\end{equation*}%
then%
\begin{eqnarray*}
-(L_{2}^{N}-L_{1}^{N})f_{2}
&=&\sum_{k=1}^{N}L_{1}^{N-k}(L_{1}-L_{2})L_{2}^{k-1}f_{2} \\
&=&\sum_{k=1}^{N}L_{1}^{N-k}(L_{1}-L_{2})f_{2}
\end{eqnarray*}%
and we have the statement.
\end{proof}

\subsubsection{Stability of the measure under perturbation of the noise}

\label{subsec:L1noise}

Now let us see that the stationary measure also varies in a Lipschitz way
with respect to perturbations of the noise: suppose $T$ satisfies Setting %
\ref{set2} and let $\rho _{1}$ and $\rho_{2}$ be two bounded variation noise
kernels and associated transfer operators $L_{i}(f)=\rho_{i}\hat{\ast}%
L_{T}(f)$ for $i=1,2 $.

\begin{lemma}
\label{noise_perturb} Let $L_i$ be defined as above, then 
\begin{equation*}
\|(L_{1}-L_{2})f\|_{L^1}\leq \|\rho _{1}-\rho _{2}\|_{L^1}\cdot \|f\|_{L^1}
\end{equation*}
\end{lemma}

\begin{proof}
Indeed, \belowdisplayskip=-12pt 
\begin{eqnarray*}
\Vert (L_{1}-L_{2})f\Vert _{L^{1}} &\leq &\Vert \lbrack \rho _{0}-\rho _{1}]%
\hat{\ast}L_{T}(f)\Vert \\
&\leq &\Vert \rho _{1}-\rho _{2}\Vert _{L^{1}}\cdot \Vert f\Vert _{L^{1}}.
\end{eqnarray*}%
\qedhere
\end{proof}

From this and the classical $L^{p}$ interpolation inequality, in the case of
dynamical systems with additive noise we get the following $L^{p}$ stability
estimate

\begin{corollary}
\label{cornovo}Let $L_{1},L_{2}:L^{1}\rightarrow L^{1}$ be two transfer
operators of deterministic systems with additive noise $L_{i}(f)=\rho _{i}%
\hat{\ast}L_{T}(f)$ for $i=1,2$ . Assume 
\begin{equation*}
\Vert L_{1}^{i}|_{V}\Vert _{L^{1}\rightarrow L^{1}}\leq C_{i}
\end{equation*}%
and suppose that $C_{N}<1$ for a certain $N$. Let $f_{i}$ be a fixed
probability measure of $L_{i}$ and Suppose $1\leq r<\infty .$ Then 
\begin{equation}
\Vert f_{2}-f_{1}\Vert _{L^{r}}\leq (\frac{\sum_{k=0}^{N-1}C_{i}}{1-C_{N}}%
\cdot \Vert \rho _{1}-\rho _{2}\Vert _{L^{1}})^{\frac{1}{r}}(2\max (||\rho
_{1}||_{BV},||\rho _{2}||_{BV}))^{1-\frac{1}{r}}.  \label{stablp}
\end{equation}
\end{corollary}

\begin{proof}
We get that for $i=1$ and $i=2$, $\Vert f_{i}\Vert _{L^{\infty }}\leq \Vert
f_{i}\Vert _{BV}\leq \max (||\rho _{1}||_{BV},||\rho _{2}||_{BV}).$ Suppose $%
1\leq r<\infty $ and $u\in L^{1}\cap L^{\infty }$, \ the classical $L^{p}$
interpolation inequality implies that $u\in L^{r}$ and%
\begin{equation*}
||u||_{L^{r}}\leq ||u||_{L^{1}}^{\frac{1}{r}}||u||_{L^{\infty }}^{1-\frac{1}{%
r}}.
\end{equation*}%
Applying this to $($\ref{eqstab}$)$%
\begin{equation}
\Vert f_{2}-f_{1}\Vert _{L^{r}}\leq (\frac{\sum_{k=0}^{N-1}C_{i}}{1-C_{N}}%
\cdot \Vert L_{1}-L_{2}\Vert _{L^{1}})^{\frac{1}{r}}(2\max (||\rho
_{1}||_{BV},||\rho _{2}||_{BV}))^{1-\frac{1}{r}}.
\end{equation}%
using the estimate for $\Vert L_{1}-L_{2}\Vert _{L^{1}}$ given in Lemma \ref%
{noise_perturb} we get $(\ref{stablp}).$
\end{proof}

\begin{lemma}
\label{rho_xi_ineq} Suppose $\rho _{1},\rho _{2}$ are the uniform kernel $%
\rho _{1}=\xi ^{-1}1_{[-\xi /2,\xi /2]}$, $\rho _{2}=\tilde{\xi}^{-1}1_{[-%
\tilde{\xi}/2,\tilde{\xi}/2]}$ we have 
\begin{equation*}
\Vert \rho _{1}-\rho _{2}\Vert _{L^{1}}\leq \frac{2}{\max \{\xi ,\tilde{\xi}%
\}}|\xi -\tilde{\xi}|.
\end{equation*}
\end{lemma}

\begin{proof}
Indeed, \belowdisplayskip=-12pt 
\begin{align*}
\Vert \rho _{1}-\rho _{2}\Vert _{L^{1}}\leq & \min \{\xi ,\tilde{\xi}%
\}\left\vert \frac{1}{\xi }-\frac{1}{\tilde{\xi}}\right\vert +\frac{|\xi -%
\tilde{\xi}|}{\max \{\xi ,\tilde{\xi}\}} \\
=& \frac{2|\xi -\tilde{\xi}|}{\max \{\xi ,\tilde{\xi}\}}.
\end{align*}%
\qedhere
\end{proof}

If $\xi ,\tilde{\xi}\geq \xi _{1}=\frac{8.73}{10^{5}}$ (see Theorem \ref%
{theo:final}) we get $\Vert \rho _{1}-\rho _{2}\Vert _{L^{1}}\leq \frac{%
2\times 10^{5}}{8.73}|\xi -\tilde{\xi}|$ and putting this in $($\ref{stablp}$)$
considering \ that in this case $||\rho _{i}||_{BV}\leq \frac{4 10^{5}}{%
8.73}+1$ we get%
\begin{equation}
\Vert f_{2}-f_{1}\Vert _{L^{r}}\leq (\frac{\sum_{k=0}^{N-1}C_{i}}{1-C_{N}}%
\cdot \frac{2\times 10^{5}}{8.73}|\xi -\tilde{\xi}|)^{\frac{1}{r}}(\frac{4\times
10^{5}}{8.73}+1)^{1-\frac{1}{r}}.  \label{withxi}
\end{equation}

We recall that the numbers $C_{i}$ \ represents the contraction rate of one
of the two operators. We show that even for these quantities we can have an
uniform estimate when $\xi ,\tilde{\xi}\geq \xi _{1}$. \ The following
result allows to estimate the $C_{i}$ constants (such that $||L_{\xi
}^{i}|_{V}||_{1}\leq C_{i}$) when the amplitude of the noise is increased.

\begin{lemma}
\label{mixatutto} \label{mono}Let $\rho _{\xi }=\xi ^{-1}1_{[-\xi /2,\xi
/2]} $ and $N_{\xi }$ the associated noise operator. If $\hat{\xi}>\xi $ and 
$||(N_{\xi }L)^{i}||_{V\rightarrow L^{1}}\leq C_{i}<1$, then

\begin{equation}
||(N_{\hat{\xi}}L)^{i}||_{V\rightarrow L^{1}}\leq C_{i}(\xi /\hat{\xi})^{i}+%
\left[ 1-(\xi /\hat{\xi})^{i}\right] <1.  \label{ec}
\end{equation}
\end{lemma}

\begin{proof}
Let $\hat\xi=\xi +\epsilon $ and $\rho_{\xi }=\xi^{-1}1_{[-\xi/2,\xi/2]}$,
we have 
\begin{align*}
\rho _{\xi+\epsilon }&=(\xi+\epsilon)^{-1} 1_{[-(\xi+\epsilon)/2
,(\xi+\epsilon)/2]} \\
=& \frac{\xi}{\xi+\epsilon}\rho_\xi + \frac{\epsilon}{\xi+\epsilon}\cdot 
\frac{1}{\epsilon} 1_{[-(\xi+\epsilon)/2,-\xi/2]\cup[\xi/2,(\xi+\epsilon)/2]}
\end{align*}

Therefore 
\begin{equation*}
N_{\xi+\epsilon}=\frac{\xi}{\xi+\epsilon}N_{\xi }+\frac{\epsilon}{%
\xi+\epsilon}M
\end{equation*}
where $M$ is the Markov operator of convolution with $\epsilon^{-1}
1_{[-(\xi+\epsilon)/2,-\xi/2]\cup[\xi/2,(\xi+\epsilon)/2]}$.

In the same way we have that $N_{\hat\xi}L$ is a convex combination of the
Markov operators $N_\xi L$ and $ML$ with coefficients $\xi/\hat\xi$ and $%
1-\xi/\hat\xi$, and 
\begin{equation*}
(N_{\hat\xi}L)^i = (\xi/\hat\xi)^i (N_\xi L)^i + \left[1 - (\xi/\hat\xi)^i %
\right] Q
\end{equation*}
for a suitable Markov operator $Q$ formed by the remaining terms of the
expansion. Considering the $L^1$ norm the inequality follows.
\end{proof}

We remark that when $C_i<1 $ it holds that the right hand of \eqref{ec} is
smaller than 1, by this we have immediately the following corollary

\begin{corollary}
\label{cormix} Let $\rho _{\xi }$ and $N_{\xi }$ as above. Suppose $%
||(N_{\xi }L)^{i}||_{V\rightarrow L^{1}}<1$, then for each $\xi \leq \hat{\xi%
}\leq 1$ it holds 
\begin{equation*}
||(N_{\hat{\xi}}L)^{i}||_{V\rightarrow L^{1}}<1
\end{equation*}%
and the system is mixing for every noise greater than $\xi $.
\end{corollary}

By continuity of the above estimates when $\hat{\xi}$ varies\ and
compactness of $[\xi _{1},1]$ if follows that $\sum_{k=0}^{N-1}C_{i}$ and $%
C_{N}$ have an uniform bound on $[\xi _{1},1]$ and then for each $r\geq 1$ \
there is $C\geq 0$ such that for each $\xi -\tilde{\xi}\in \lbrack \xi
_{1},1]$ 
\begin{equation}
\Vert f_{2}-f_{1}\Vert _{L^{r}}\leq C|\xi -\tilde{\xi}|^{\frac{1}{r}}
\label{holdermeas}
\end{equation}%
proving the H\H{o}lder stability of the stationary measure in $L^{r}.$

\subsection{Stability of the Lyapunov exponent}

Now we see how from the H\H{o}lder stability of the stationary measure
proved in the previous section we can easilly deduce the H\H{o}lder
stability of the Lyapunov exponent $\lambda _{\xi }$ and the H\H{o}lder
continuity of $\lambda _{\xi }$ as $\xi $ varies. \ We recall that $\lambda
_{\xi }:=\int_{0}^{1}\phi (x)d\mu _{\xi }$ where $\phi (x)=\log |T^{\prime
}(x)|.$

\begin{corollary}
\label{holdlyapcor}For each $r>1$ we have that \ $\phi \in L^{\frac{r}{r-1}}$
and there is $C\geq 0$ such that for each $\xi ,\tilde{\xi}\in \lbrack \xi
_{1},1]$,%
\begin{equation*}
|\lambda _{\xi }-\lambda _{\tilde{\xi}}|\leq ||\phi ||_{L^{\frac{r}{r-1}%
}}C|\xi -\tilde{\xi}|^{\frac{1}{r}}.
\end{equation*}
\end{corollary}

\begin{proof}
We can get an explitic formla for $\phi $, indeed%
\begin{equation*}
T^{\prime }(x)=\left\{ 
\begin{array}{c}
-\frac{1}{6}\frac{e^{-x}}{\left( 8x-1\right) ^{\frac{2}{3}}}\left(
24x+6a\left( 8x-1\right) ^{\frac{2}{3}}-11\right) ~~~~~~0\leq x\leq 0.3 \\ 
-\frac{19\times 10^{19}}{3}cx^{18}e^{-\frac{190}{3}x}\left( 10x-3\right)
~~~0.3<x\leq 1%
\end{array}%
\right. 
\end{equation*}%
by this $\phi \in L^{p}[0,1]$ for each $p\in \lbrack 1,\infty )$ and then by
the holder inequality%
\begin{eqnarray*}
|\lambda _{\xi }-\lambda _{\tilde{\xi}}| &\leq &\int_{0}^{1}\phi
(x)[f_{1}-f_{2}]dm \\
&\leq &||\phi ||_{L^{\frac{r}{r-1}}}\Vert f_{2}-f_{1}\Vert _{L^{r}}
\end{eqnarray*}%
from which the statemet follows applying $($\ref{holdermeas}$).$
\end{proof}




\section{Appendix: Generalities, ergodicity and Lyapunov exponents in
Random  Dynamics}


\label{sec:rds} In this section we recall some basic results and definitions
in the ergodic theory of random transformations and the integral formula for
the Lyapunov exponent. We refer to \cite[Chapter 5]{Viana}; another
classical reference is \cite{Ar}.

Let $X$ be the interval $[-\xi,\xi]$ endowed with the Borel $\sigma$-algebra
and $p$ the uniform probability on $X$; let $\Xi=X^{\mathbb{N}}$ the space
of sequences with values in this space endowed with the product $\sigma$%
-algebra $\Omega$ and the product measure $\mathbb{P}=p^{\mathbb{N}}$. Let $%
\phi$ be the shift acting on $\Xi$.

We endow the interval $[0,1]$ with the Borel $\sigma$-algebra $\beta$ and we
define the measurable skew product: 
\begin{equation*}
F:(\Xi,\Omega)\times([0,1],\beta)\to (\Xi,\Omega)\times ([0,1],\beta)\quad
F(\omega,x)=(\phi (\omega ), T(x)+(\omega)_0).
\end{equation*}

This skew product models the evolution of the stochastic process 
\begin{equation*}
X_{n+1}=T(X_{n})+\omega _{n}.
\end{equation*}%
where $\omega _{n}$ is a sequence of i.i.d. random variables uniformly
distributed in $[-\xi ,\xi ]$ endowed with the Borel $\sigma $-algebra.

In the following, let $L_{\xi}$ be the annealed transfer operator as defined
in Section \ref{map}. This operator embodies how measures behave ``in
average'' under the action of the random dynamical system; let $\mu$ be a
measure on $[0,1]$ and $T_{\epsilon}(x)=T(x)+\epsilon$, we have that 
\begin{equation*}
L_{\xi}\mu(B)=\int \mu(T_{\epsilon}^{-1}(B))dp(\epsilon).
\end{equation*}

Since $\phi $ is the one sided shift and $\nu $ is a stationary measure,
i.e., $L_{\xi }\nu =\nu $, the product measure $\mathbb{P}\times \nu $ is
invariant for $F$ \cite[Proposition 5.4]{Viana}. Since the transfer operator
related to convolution with a Bounded Variation kernel is regularizing from $%
L^{1}$ to $BV$ which is compactly immersed in $L^{1}$ it is easy to see that
the transfer operator has at least one stationary measure $f_\xi$ with density in $BV
$ (see \cite{GG} Lemma 23 for more details). 

\begin{definition}
A stationary measure $\nu$ is said to be ergodic if the measure $\mathbb{P}%
\times\nu$ is ergodic for $F$.
\end{definition}

\begin{remark}
While we use this as our definition of ergodicity for the random dynamic, we
refer to \cite{Viana} for equivalent alternative definitions.
\end{remark}

\begin{proposition}
\label{ergo} Let $L_{\xi}$ be the transfer operator associated to the
Belousov-Zhabotinsky map with additive noise of size $\xi$. Let $%
\xi_1=0.873\times 10^{-4}$. For each $\xi\geq \xi_1$ there exists a unique,
ergodic stationary measure $\mu_{\xi}$ for the operator $L_{\xi}$, and for $%
\mathbb{P}\times\mu_{\xi}$ almost every point $(\omega,x)$, we have that 
\begin{equation*}
\lim_{n\to +\infty} \frac{1}{n}\sum_{i=0}^{n-1}\log|T^{\prime
i}(\omega,x))|=\int_0^1 \log(|T^{\prime }(x)|)d\mu_{\xi}(x)=:\lambda_{\xi}
\end{equation*}
\end{proposition}

\begin{proof}
The operator $L_{\xi }$ is defined from $SM(X)$ to $L^{1}([0,1])$;
therefore, all of its fixed points belong to $L^{1}([0,1])$.

In the last row of Table \ref{t1} the columns $n$ and $\alpha$ shows that
for the noise of amplitude $\xi_1$ it holds 
\begin{equation}  \label{055}
||L_{\xi_1}^{75}||_{V\to L^1}\leq 0.55
\end{equation}
(the $n$ in the table is valid for both the approximation $L_{\delta, {\xi_1}
}$ and the original operator $L_{\xi_1}$, as explained after Lemma \ref%
{cflemma}).

Suppose $f,g$ are two fixed probability measures for the operator, by %
\eqref{055} : 
\begin{equation*}
||f-g||_{L^1}=||L_{\xi_1}^{75}(f-g)||_{L^1}\leq 0.55||f-g||_{L^1} 
\end{equation*}
which implies that $||f-g||_{L^1}=0$. Thus $L_{\xi_1}$ has a unique fixed
probability measure $\mu_{\xi_1}$ in $L^1$, i.e., the only stationary
measure is $\mu_{\xi_1}$. Now the same holds for every $\xi\geq \xi_1$
thanks to Corollary \ref{cormix}. The corollary indeed implies that $%
||L_{\xi}^{75}||_{V\to L^1}<1$ and we can repeat the same reasoning as above,
obtaining an unique stationary measure for each such $L_{\xi}$.

A kind of Ergodic Decomposition Theorem is established for stationary
measures \cite[Theorem 5.13]{Viana}; the theorem states that every
stationary measure can be written as a convex combination of ergodic
stationary measures. Since $\mu_{\xi}$ is the unique stationary measure, it
follows that $\mu_{\xi}$ is an ergodic stationary measure. Therefore $%
\mathbb{P}\times \mu_{\xi}$ is an ergodic measure for $F$ and the result
follows applying Birkhoff Ergodic Theorem to $F$.
\end{proof}



\section{Appendix: operator norms and variation estimates}
\label{appendix2}


\label{appendix} In this section we prove several technical lemmas and
estimates about operator norms and variation of iterates of measures which
are used in Section \ref{firstalgo}.

\subsection{Operator norms}

The following Lemma  allows to estimate the variation of a  boundary reflecting convolution.

\begin{lemma}
\label{chepallepero} Let $\rho_{\xi}(x)$ be a real function with bounded variation
and support contained in $(-\xi,\xi)$. Let $b(x)$ a function with zero
average on the unit interval. We have 
\begin{equation*}
\mathrm{Var}( \rho_{\xi}\hat{\ast} f) \leq \mathrm{Var}(\rho_{\xi})\cdot \Vert f\Vert _{L^1([0,1])}.
\end{equation*}
As a consequence we have that 
\begin{equation*}
\Vert N_\xi \Vert_{L^1\rightarrow \mathrm{Var}}\leq \mathrm{Var}%
(\rho_\xi)=\xi^{-1}\mathrm{Var}(\rho).
\end{equation*}
\end{lemma}


\begin{proof}
Since the support of $\hat{f}$ is contained in $[0,1]$ we have that
the support of $\rho_{\xi}\ast \hat{f}$ is contained in $[-\xi,1+\xi]$ and
that:
\[
\rho_{\xi}\hat{\ast}f(x)=\rho_{\xi}\ast \hat{f}(x)+\rho_{\xi}\ast \hat{f}(-x)+\rho_{\xi}\ast \hat{f}(2-x). 
\]

We recall that:
\[
\Var_{[-1,2]} (\rho_{\xi}\ast \hat{f})\leq \Var(\rho_{\xi})\cdot ||\hat{f}||_{L^1([-1,2])}=\Var(\rho_{\xi})\cdot ||f||_{L^1([0,1])};
\]
which implies
\belowdisplayskip=-12pt 
\begin{align*}
\Var_{[0,1]}(\rho_{\xi}\hat{\ast}f)&\leq \Var_{[0,1]}(\rho_{\xi}\ast \hat{f})+\Var_{[-1,0]}(\rho_{\xi}\ast \hat{f})+\Var_{[1,2]}(\rho_{\xi}\ast \hat{f})\\
&\leq \Var(\rho_{\xi})\cdot ||f||_{L^1([0,1])}.
\end{align*}
\end{proof}

The following Lemma is a small improvement (by a factor $2$) of a Lemma
which has already been used in \cite{L} and \cite{GN}.
While it seems to be folklore, we prove it here for a matter of completeness.

\begin{lemma}
\label{Lemma3} For each $f\in BV$ 
\begin{equation*}
||f-\pi _{\delta }f||_{L^1}\leq \frac{\delta }{2}\mathrm{Var}(f).
\end{equation*}%
In other words%
\begin{equation*}
\Vert 1-\pi _{\delta }\Vert _{\mathrm{Var}\rightarrow L^{1}}\leq \delta /2.
\end{equation*}
\end{lemma}
\begin{proof}
Let $I$ be an interval of the partition $\Pi$, and assume $f$ to have variation $v$ in
  $I$. In $I$, $\pi_\delta(f)$ is constant and equal to the average of $f$ in $I$.
 
Subtracting a constant we can assume $f$ to have average $0$ in $I$, so that we will be
estimating $ \|f\|_{L^1(I)} $
assuming $\pi_\delta f$ to be $0$ in $I$. We will disregard the variation at the
boundary of $I$, and will only assume a bound on the variation on the interior.
  
We start by supposing $f$ to be piecewise constant, with two pieces:
$f$ is $-x$ on an interval of size $a$ and $v-x$
in an interval of size $\delta-a$. Since the average is $0$ we have that 
\begin{equation*}
xa = (v-x)(\delta-a),
\end{equation*}
that implies $x = v-av/\delta$. The $L^1$ norm in $I$ is 
\begin{equation*}
2xa = 2a(v-av/\delta),
\end{equation*}
and has derivative with respect to $a$ equal to 
\begin{equation*}
2v - 4av/\delta,
\end{equation*}
which becomes zero for $a = \delta/2$. Consequently $x=v/2$, and a variation 
$v$ contributed an $L^1(I)$ norm of $2ax=v\delta/2$.

We claim that the biggest ratio $\|f\|_{L^1(I)} / v$ is attained when $f$ is piecewise
constant attaining exactly the values $-x$ and $v-x$. Indeed, if this was not the case,
we could build a new function $\tilde{f}$ selecting the region where $f$ is non-negative
(or non-positive), and setting as value the average of $f$ in this region. In this way
we obtain a $\tilde f$ that has the same $L^1(I)$-norm, but smaller variation. If these
two regions are not a partition of $I$ in two intervals, then the difference between the
maximum and the minimum of $f$ is smaller than $v$, and again the $f$ is not optimal.

Applying this estimate to all the intervals of the partition we have that a
total variation of $v$ can give a total $L^{1}$ norm of $v\delta /2$, and
consequently we have the lemma.
\end{proof}

By Lemma \ref{Lemma3}, since $\Vert N_{\xi }\Vert _{L^{1}\rightarrow \mathrm{%
Var}}=\xi ^{-1}\mathrm{Var}(\rho )$ by Lemma \ref{chepallepero} we get:

\begin{corollary}
\label{c1}With the notations defined above, we have 
\begin{equation*}
\Vert (1-\pi _{\delta })N_{\xi }\Vert _{L^{1}\rightarrow L^{1}}\leq \frac{1}{%
2}\delta \xi ^{-1}\mathrm{Var}(\rho ).
\end{equation*}
\end{corollary}

This corollary is used in Sec. \ref{init}.

\subsubsection{Estimate for $\Vert N_{\protect\xi }(1-\protect\pi _{\protect%
\delta })\Vert _{L^{1}}$}

To estimate this item (necessary in Sec. \ref{init}) we will use the $W$
norm, defined in definition \ref{def:Wasserstein}.

\begin{proposition}
\label{prosop}We have 
\begin{equation*}
\Vert N_{\xi }(1-\pi _{\delta })\Vert _{L^{1}\rightarrow L^{1}}\leq \frac{1}{%
2}\delta \xi ^{-1}\mathrm{Var}(\rho ).
\end{equation*}
\end{proposition}

The proof will be postponed to the following lemmas. The first lemma relates
the convolution with the $\Vert ~\Vert _{W}$ norm. 

\begin{lemma}
\label{L1}Let $a(x)$ be a real function with bounded variation with support
contained in $(-1/2,1/2)$, and $b(x)$ supported in $[0,1]$ and with zero average. 
We have 
\begin{equation*}
\Vert a\ast b\Vert _{L^{1}([-1,2])}\leq \mathrm{Var}(a)\cdot \Vert b\Vert _{W}.
\end{equation*}
As a consequence we have that 
\begin{equation*}
\Vert a\hat{\ast} b\Vert _{L^{1}([0,1])}\leq \mathrm{Var}(a)\cdot \Vert b\Vert _{W}
\end{equation*}
and therefore
\begin{equation*}
\Vert N_\xi \Vert_{W\rightarrow L^1}\leq \mathrm{Var}(\rho_\xi)=\xi^{-1}%
\mathrm{Var}(\rho).
\end{equation*}
\end{lemma}

\begin{proof}
Let's prove the lemma assuming first that $a(x)$ is absolutely continuous. Let
$B(x)=\int_{0}^{x}b(t)\mathrm{d}t$; integrating by parts we have
\begin{align*}
(a\ast b)(x)& =\int_{-1}^{1}a(-t)b(x+t)\mathrm{d}t \\
& =[a(-t)B(x+t)]_{-1}^{+1}-\int_{-1}^{1}-a^{\prime
}(-t)B(x+t)\mathrm{d}x,
\end{align*}%
the boundary term being $0$ because $b(t)$ is zero-average in the interval. 

The support of $a\ast b$ is contained in $[-1/2,3/2]$;
we compute now:
\begin{align*}
\Vert a\ast b\Vert _{L^{1}([-1,2])}& =\int_{-1}^2 \left\vert \int_{-1}^{1
}a^{\prime }(t)B(x-t)\mathrm{d}t\right\vert \mathrm{d}x \\
&\leq\int_{-1}^2  \int_{-1}^{1
}|a^{\prime }(t)B(x-t)| \mathrm{d}t\, \mathrm{d}x\\
&=\int_{-1}^1  \int_{-1-t}^{2-t
}|a^{\prime }(t)B(u)| \mathrm{d}u\,dt\\
&\leq \int_{-1}^1 |a^{\prime }(t)|\mathrm{d}t \int_{0}^{1} |B(u)| \mathrm{d}u
%
\end{align*}
(putting $u=t-x$ and using that $B(u)$ has support in $[0,1]$)
\begin{align*}
& \leq \int_{-1 }^{1 }|a^{\prime }(t)|\mathrm{d}t\cdot \int_{0}^1 |B(u)|%
\mathrm{d}u \\
& \leq \mathrm{Var(a)}\cdot \Vert B\Vert_{L^{1}} \\
& \leq \mathrm{Var(a)}\cdot \Vert b\Vert_{W}.
\end{align*}
When $a(x)$ is not absolutely continuous, let's just choose absolutely continuous functions
$a_n$ such that $a_n\rightarrow a$ in $L^1$ and $\Var(a_n)\rightarrow \Var(a)$, and apply Fatou's Lemma.

Now, observing as before that:
\[
\rho_{\xi}\hat{\ast}f(x)=\rho_{\xi}\ast \hat{f}(x)+\rho_{\xi}\ast \hat{f}(-x)+\rho_{\xi}\ast \hat{f}(2-x). 
\]
we have that
\[
||\rho_{\xi}\hat{\ast}f||_{L^1([0,1])}\leq ||\rho_{\xi}\ast \hat{f}||_{L^1([-1,2])}\leq
\Var(\rho_{\xi})||f||_W.
\qedhere
\]
\end{proof}

To prove Prop. \ref{prosop} we also need a bound for $\Vert 1-\pi _{\delta
}\Vert _{L^{1}\rightarrow W}$.

\begin{lemma}
\label{lem1mpLW} For the Ulam discretization of size $\delta $ we
have 
\begin{equation*}
\Vert 1-\pi _{\delta }\Vert _{L^{1}\rightarrow W}\leq \delta /2.
\end{equation*}
\end{lemma}

\begin{proof}
  We will prove the lemma is true on the space of all measures in the interval. Then we
  can view any measure as a combination of point masses, using that $1-\pi_\delta$ is a linear
  operator on signed measures, and $W$ is a norm on zero-average measures.

  Let $\Delta_t$ be the atomic measure centered in $t$ and with weight $1$ (Kronecker's
  $\delta_t$, we use capital $\Delta$ to avoid confusion), then
  $(1-\pi_\delta)\Delta_t = \Delta_t-\delta^{-1}\chi_I$, where $I=(p_i,p_{i+1})$ is the
  interval of the $\delta$-sized partition containing $t$. To compute its $W$-norm we need
  to compute the $L^1$ norm of
\begin{align*}
  u_t(x) &= \int_0^x \left(\Delta_t(y) - \delta^{-1}\chi_I(y)\right)  \mathrm{d}y \\
         &= \left\{ \begin{array}{cl}
                      \delta^{-1}(p_i-x) & \text{for }x\in[p_i,t], \\
                      \delta^{-1}(p_{i+1}-x) & \text{for }x\in[t,p_{i+1}], \\
                      0 & \text{elsewhere}.
                    \end{array} \right.
\end{align*}
Its $L^1$ norm is computed as 
\begin{align*}
 \|u_t\|_{L^1} &= \delta^{-1}\int_{p_i}^{t} (x-p_i) \mathrm{d} x 
               + \delta^{-1}\int_{t}^{p_{i+1}} (p_{i+1}-x) \mathrm{d} x \\
              &= \delta^{-1}\frac{1}{2}(t-p_i)^2 + \delta^{-1}\frac{1}{2}(p_{i+1}-t)^2 \\
              &\leq \frac{\delta}{2}
\end{align*}
because it is a quadratic function that reaches its maximum for $t\in\{p_i,p_{i+1}\}$,
where its value is exactly $\delta/2$.

Now, the general case. We are applying a linear operator and a norm on measures,
consequently we are applying a weak-$\ast$ lower semi-continuous function. We verified
that such function is $\leq \delta/2$ on atomic measures, and this holds for all finite
combinations of atomic measures. Since finite combinations of atomic measures are
weak-$\ast$ dense in the space of all measure we have the lemma.
\end{proof}

\begin{proof}[Proof of Prop. \ref{prosop}]
By Lemma \ref{L1} and Lemma \ref{lem1mpLW} we have
\[
   \|N_\xi(1-\pi_\delta)\|_{L^1} \leq
   \|1-\pi_\delta\|_{L^1\rightarrow{}W} \cdot \|N_\xi\|_{W\rightarrow L^1} \leq
   \frac{1}{2}\delta\xi^{-1}\Var(\rho). \qedhere
\]
\end{proof}

\subsubsection{An estimate for $\Vert (1-\protect\pi _{\protect\delta %
})|_{X}\Vert _{\mathrm{Var}\rightarrow W}$}

\begin{lemma}
\label{idmpVarW} We have 
\begin{equation*}
\|(1-\pi_\delta)\|_{\mathrm{Var}(X)\rightarrow W(X)} \leq \frac{\delta^2}{8}.
\end{equation*}
for each $X\subseteq [0,1]$ that is a union of intervals of the partition.
\end{lemma}

\begin{proof}
  Let $g\in L^{1}$, let us estimate $\Vert (1-\pi _{\delta })g\Vert _{W}$; assume $g$ to
  have support contained in an interval of the partition $I$, and subtracting
  $\pi _{\delta }g$ assume its average in $I$ to be $0$, and as usual let $G=\int
  g$. Assume the variation to be $v$, and the maximum of $G'$ be $%
  u$.
  Then we have $u-v\leq G^{\prime }\leq u$. Assume $a<b\in I$ are points such that
  $G(a)=G(b)=0$ and $G$ is nonnegative (the nonpositive case being symmetrical),
  then $G$ is bounded by the functions
\begin{equation*}
u(x-a),\qquad (v-u)(b-x).
\end{equation*}%
These linear functions form a triangle of height $h$ that satisfies 
\begin{equation*}
h/u+h/(v-u)=b-a\leq \delta ,
\end{equation*}%
and therefore $h\leq \delta u(v-u)/v$. An estimate for the integral of $G$
over $[a,b]$ is therefore obtained multiplying by $(b-a)/2$, and its
integral over $I$ is therefore bounded by 
\begin{equation*}
\frac{\delta ^{2}u(v-u)}{2v}.
\end{equation*}%
Deriving with respect to the parameter $u$ as usual, we have that the
maximum is attained for $u=v/2$, and is equal to $v\delta ^{2}/8$. Since $v$
was the variation in the interval of the partition, we proved the lemma.
\end{proof}

\begin{remark}
In the proof we just used the total variation of $g$ in the \emph{interiors}
of the intervals $I\in\Pi$, disregarding the possible jumps at the boundary
between different intervals of the partition.
\end{remark}

\subsubsection{An estimate for $\Vert L\Vert _{W(I)\rightarrow W}$}

We give here an estimate for $\Vert L\Vert _{W(I)\rightarrow W}$, that is
the $W$ norm of $Lf$ for a function (having zero average and) whose support
is contained in an interval $I$. This estimate is required in the proof of
Lemma \ref{lemma_NpL1pf}.

\begin{lemma}
\label{LemmaW}We have 
\begin{equation*}
\Vert L\Vert _{W(I)\rightarrow W}\leq \Vert T^{\prime }\Vert _{L^{\infty
}(I)}.
\end{equation*}
\end{lemma}

\begin{proof}
Observe that for each function $h\in L^{1}$, and putting $H=\int h$ so that $%
\Vert h\Vert _{W}=\Vert H\Vert _{L^{1}}$, we have 
\begin{align*}
  \int Lh& =\sum_{i}\int_{0}^{x}\frac{h(T_{i}^{-1}(t))}{T^{\prime
           }(T_{i}^{-1}(t))}\mathrm{d}t \\
         & =\sum_{i}\int_{0}^{T_{i}^{-1}(x)}h(y)\mathrm{d}y\qquad (\text{via }%
           t=T_{i}(y)) \\
         & =\sum_{i}H(T_{i}^{-1}(x))
\end{align*}
(we assume the above integral to be extended as a constant value for $x$ outside the image
$T_i$).

Therefore
\belowdisplayskip=-12pt 
\begin{align*}
\Vert Lh\Vert _{W}& =\left\Vert \int Lh\right\Vert _{L^{1}} \\
& \leq\sum_{i}\int_{0}^{1}|H(T_{i}^{-1}(x))|\mathrm{d}x \\
& \leq\sum_{i}\int_{T_{i}^{-1}([0,1])}|H(y)|\cdot T^{\prime }(y)\mathrm{d}%
y\qquad (\text{via }x=T_{i}(y)) \\
& \leq\int_{0}^{1}|H(y)|\cdot T^{\prime }(y)\mathrm{d}y \\
& \leq \Vert H\Vert _{L^{1}}\cdot \Vert T^{\prime }\Vert _{L^{\infty }(%
\mathrm{Supp}(H))} \\
& \leq \Vert h\Vert _{W}\cdot \Vert T^{\prime }\Vert _{L^{\infty }(\mathrm{%
Supp}(I))}.
\end{align*}
\end{proof}

\subsection{Variation estimates}

In this section we collect several variation estimates which are used in
Section~\ref{firstalgo}.

\subsubsection{The variation of $Lg$ in an interval $I$}

\label{varLg}

We recall here how we can estimate the variation of $Lg$ in an interval, the
estimate will be used computationally for estimating the variation of $L%
\tilde{f}$ in intervals of some partition. This estimate is used to compute
the bound provided by Lemma \ref{lemma_N1pLf}.

\begin{lemma}[Local variation inequality]
\label{varLf} Let $I\subseteq [0,1]$ be an interval. Let $L_{i}g$ be the
component of $Lg$ coming from the $i$-th branch, defined in (\ref{Li}). We
have $\mathrm{Var}_{I}(Lg)\leq \sum_{i}\mathrm{Var}_{I}(L_{i}g),$ and the
variation of each component can be estimated as 
\begin{align*}
\mathrm{Var}_{I}(L_{i}g)\leq& \mathrm{Var}_{T_{i}^{-1}(I)}(g)\cdot
\left\Vert \frac{1}{T^{\prime }}\right\Vert _{L^{\infty
}(T_{i}^{-1}(I))}+\Vert g\Vert_{L^{1}(T_{i}^{-1}(I))}\cdot \left\Vert \frac{%
T^{\prime\prime}}{T^{\prime 2}}\right\Vert _{L^{\infty }(T_{i}^{-1}(I))} \\
&+\sum_{y\in \partial Dom(T_i):T(y)\in I}\left|\frac{g(y)}{T^{\prime }(y)}%
\right|.
\end{align*}
\end{lemma}

\begin{proof}
Let $g\in C^1([0,1])$, the bounded variation case following by density of $C^1$ in $BV$:
\begin{equation*}
(L_ig)^{\prime }(x)=\frac{g^{\prime }(T_{i}^{-1}(x))}{%
T^{\prime}(T_i^{-1}(x))^{2}}-\frac{g(T_{i}^{-1}(x))\cdot T^{\prime \prime
}(T_{i}^{-1}(x))}{T^{\prime }(T_{i}^{-1}(x))^{3}}.
\end{equation*}%
And consequently the variation of $g$ in an interval $I$ can be
bounded as
\begin{align*}
\int_{I}\left\vert \frac{g^{\prime }(T_{i}^{-1}(x))}{%
T^{\prime}(T_i^{-1}(x))^{2}}\right\vert & +\left\vert \frac{g(T_{i}^{-1}(x))
\cdot T^{\prime \prime }(T_{i}^{-1}(x))}{T^{\prime
}(T_{i}^{-1}(x))^{3}}\right\vert \mathrm{d}x= \\
\int_{T_i^{-1}(I)}\left\vert \frac{g^{\prime }(y)}{T^{\prime }(y)}%
\right\vert & +\left\vert \frac{g(y)\cdot T^{\prime \prime }(y)}{%
T^{\prime 2}}\right\vert \mathrm{d}y
\end{align*}%
replacing $T^{-1}(x)$ by $y$, and $\mathrm{d}x$ by $T^{\prime }(y)\mathrm{d}%
y $,
as usual. 

Taking into account the value of $L_ig$ at the boundary
of its support we obtain the estimate.
Please remark that if $T_i$ is not full branch the support of $L_i g$
is strictly contained in $[0,1]$
\end{proof}
\subsubsection{The variation of $N_{\protect\xi }g$}

\label{varNLf}

We deduce an algorithm for estimating $\mathrm{Var}_{I}(N_{\xi }g)$ provided
that we have enough information about $g$.

\begin{lemma}
\label{Lemma6} Assume $\rho _{\xi }=\xi ^{-1}\chi_{[-\xi /2,\xi /2]}(x)$.
For any interval $I=[a,b]$ we have 
\begin{equation*}
\mathrm{Var}_{I}(N_{\xi }g)\leq \xi ^{-1}\cdot \min \left\{ \Vert g\Vert
_{L^{1}(I-\xi /2 \cup I+\xi /2)}Var(\rho ),|I| \cdot \mathrm{Var}_{[a-\xi /2,b+\xi
/2]}(g)||\rho ||_\infty \right\} .
\end{equation*}
Hence in the case where $\rho _{\xi }=\xi^{-1}\chi_{[-\xi /2,\xi /2]}$
\begin{equation*}
\mathrm{Var}_{I}(N_{\xi }g)\leq \xi ^{-1}\cdot \min \left\{ \Vert g\Vert
_{L^{1}(I-\xi /2 \cup I+\xi /2)},|I| \cdot \mathrm{Var}_{[a-\xi /2,b+\xi
/2]}(g)\right\} .
\end{equation*}

\end{lemma}

\begin{proof}
  For an interval of the partition $I$ (centered in $t$, say), let us set
  \[
    \overline{I}=I-\xi /2 \cup I+\xi /2,
  \]
 we have 
\begin{align*}
\int_{I}|(\rho _{\xi }\ast g)^{\prime}|\mathrm{d}x& \leq \int_{I}\left( \int_{\overline{I}}
|\rho _{\xi }^{\prime }(x-y)g(y)|\mathrm{d}y\right) \mathrm{d}x \\
& \leq \int_{\overline{I}} |g(y)|\int_{I}|\rho _{\xi }^{\prime }(x-y)|\mathrm{d}x\mathrm{d}y
\\
& \leq \int_{\overline{I}} |g(y)|\Var(\rho _{\xi })\mathrm{d}y \\
& \leq \xi ^{-1}\Var(\rho )\cdot \big(\Vert g\Vert _{L^{1}(I-\xi /2)}+\Vert g\Vert
_{L^{1}(I+\xi /2)}\big).
\end{align*}

Symmetrically putting $I=[a,b]$ we have 
\begin{align*}
\int_{I}|(\rho _{\xi }\ast g)^{\prime}|\mathrm{d}x& \leq \int_{I}\left( \int
|\rho _{\xi }(x-y)g^{\prime }(y)|\mathrm{d}y\right) \mathrm{d}x \\
& \leq \int_{I} ||\rho_\xi||_\infty  \int_{x-\xi /2}^{x+\xi /2}  |g^{\prime }(y)|\mathrm{d}y\mathrm{d%
}x \\
& \leq \int_{I}\xi ^{-1} ||\rho_\xi||_\infty  \int_{a-\xi /2}^{b+\xi /2}  |g^{\prime }(y)|\mathrm{d}y\mathrm{d%
}x \\
& \leq \int_{I}\xi^{-1}||\rho_\xi||_\infty \mathrm{d}x \cdot \int_{a-\xi /2}^{b+\xi /2}  |g^{\prime }(y)|\mathrm{d}y \\
& \leq |I| \xi ^{-1}||\rho_\xi||_\infty \Var_{[a-\xi /2,b+\xi /2]}(g). \qedhere
\end{align*}
\end{proof}

\begin{remark}
If the minimum was always obtained as the first part, we end up estimating
the total variation of $N_{\xi}g$ as $2\xi^{-1}=\mathrm{Var}(\rho_{\xi})$,
getting the same bound as in the a-priori estimate. If the second part is
always bigger, the estimate is approximatively $\mathrm{Var}(g)$, with a
small increase due to the fact that we will be integrating the variation
over an interval of size $\xi +\delta $ rather than $\xi $.
\end{remark}



\end{document}